%% file: main.tex
\documentclass[11pt]{article}

\usepackage{jmlr2e}
\usepackage{multirow} 
\usepackage{FMTmicros}
\usepackage{tikz}
\usepackage{subfigure}
\usetikzlibrary{arrows,decorations.pathmorphing,backgrounds,positioning,fit,petri,graphs}
\usepackage[labelfont=bf]{caption}
\usepackage[labelfont=scriptsize]{subcaption}
\usepackage{booktabs}
\usepackage{appendix}
\usepackage[linesnumbered,lined,boxed,commentsnumbered,ruled]{algorithm2e}
\usepackage{enumerate}
\usepackage{placeins}
\definecolor{citecol}{HTML}{6F130C}
\definecolor{tableofcontent}{HTML}{1F4A83}
\definecolor{urlcol}{HTML}{2470D8}

\usepackage{hyperref}
\hypersetup{
    colorlinks=true,       
    linkcolor=tableofcontent,
    citecolor=citecol,        
    urlcolor=blue,           
}
\usepackage{graphics}
\usepackage{longtable}
\usepackage{threeparttable}
\usepackage{xcolor}
\usepackage{colortbl}
\usepackage{graphicx}
\usepackage{booktabs}       
\usepackage{amsfonts}       
\usepackage{nicefrac}       
\usepackage{microtype}      
\usepackage{callouts}

\input{math_commands.tex}

\newcounter{jyhcomm}
\definecolor{aqua}{rgb}{0.00,0.67,0.80}

\usepackage{lastpage}

\jmlrheading{27}{2026}{1-\pageref{LastPage}}{7/24; Revised 1/26}{5/26}{24-1075}{Yuanhong Jiang, Dongmian Zou, Xiaoqun Zhang, Yu Guang Wang}
\ShortHeadings{Deep Scattering Message Passing}{Jiang, Zou, Wang, Zhang}
\firstpageno{1}

\begin{document}

\title{Limiting Over-Smoothing and Over-Squashing of Graph Message Passing by Deep Scattering Transforms}

\author{\name Yuanhong Jiang
\email \texttt{\rm jyh36681@gmail.com}\\
\addr 
Shanghai Jiao Tong University
\AND
\name  Dongmian Zou
\email \texttt{\rm dongmian.zou@duke.edu}\\
       \addr Duke Kunshan University
\AND
\name Xiaoqun Zhang$^{*}$
\email \texttt{\rm xqzhang@sjtu.edu.cn}\\
       \addr 
       Shanghai Jiao Tong University
\AND     
\name Yu Guang Wang$^{*}$
\email \texttt{\rm yuguang.wang@sjtu.edu.cn}\\
       \addr 
       Shanghai Jiao Tong University\\
       \&
       University of New South Wales}

\renewcommand{\thefootnote}{\fnsymbol{footnote}} 
\footnotetext[1]{ Corresponding author. This work was supported by the NSFC (Grant No. 12090024, 12301117), the Natural Science
Foundation of Chongqing, China (CSTB2023NSCQ-LZX0054). We thank the Student Innovation Center at Shanghai Jiao Tong University for providing us the computing services.} 
\editor{Joan Bruna}

\maketitle

\begin{abstract}
Graph neural networks (GNNs) have become pivotal tools for processing graph-structured data, leveraging the message passing scheme as their core mechanism. However, traditional GNNs often grapple with issues such as instability, over-smoothing, and over-squashing, which can degrade performance and create a trade-off dilemma. In this paper, we introduce a discriminatively trained, multi-layer Deep Scattering Message Passing (DSMP) neural network designed to overcome these challenges. By harnessing spectral transformation, the DSMP model aggregates neighboring nodes with global information, thereby enhancing the precision and accuracy of graph signal processing. We provide theoretical proofs demonstrating the DSMP's effectiveness in mitigating these issues under specific conditions. Additionally, we support our claims with empirical evidence and thorough frequency analysis, showcasing the DSMP's superior ability to address instability, over-smoothing, and over-squashing.
\end{abstract}

\vspace{0.5cm}
\begin{keywords}
Message Passing Neural Network, Scattering Transform, Over-smoothing, Over-squashing
\end{keywords}

\tableofcontents

\newpage
\section{Introduction}
Graphs offer an elegant and comprehensive means to describe data's intricate topology or structure. A wide array of relational systems and structured entities have been aptly represented by graphs, encompassing their internal connections. Notable examples of such applications include recommendation systems~\citep{wu2022graph}, prediction of molecular properties~\citep{bumgardner2021drug}, analysis of social networks~\citep{min2021stgsn}, and particle system investigations~\citep{li2009optimal,wang2022acmp}. The inherent topology embedded within these graphs enables the establishment of local communication channels, facilitating the exchange of neighborhood information~\citep{bruna2013spectral,kipf2017semi,velivckovic2017graph}. 
In this paper, we develop a scattering wavelet-based message passing neural network, which has proved excellent performance in dealing with the commonly prevalent issues of over-smoothing, over-squashing, and instability in GNNs.

GNNs typically adhere to the \emph{message passing} paradigm, which is the foundation of \emph{message passing neural networks} (MPNNs)~\citep{gilmer2017neural}. Most MPNNs extract graph embeddings from the spatial domain and have demonstrated their utility in diverse real-world tasks. These networks efficiently capture the complex relationships and interactions within graphs, making them useful tools for understanding and analyzing structured data.

Despite the remarkable success of MPNNs in various real-world applications, they still encounter significant challenges with depth scalability and information bottleneck. One major issue is the over-smoothing phenomenon, which occurs when node features gradually converge and lose their uniqueness as the number of layers increases~\citep{li2018deeper}. This convergence not only diminishes the distinctiveness of node representation but also directly leads to a decline in  performance~\citep{chen2020measuring}. The severity of over-smoothing is often quantified by analyzing the \textit{Dirichlet energy} of the embedded graph data features~\citep{cai2020note,di2022graph}.

Another notable challenge is the over-squashing problem, originally identified by \cite{alon2020bottleneck}. This arises when node features are unable to effectively capture information from distant nodes, resulting in a loss of crucial contextual details. \citet{di2023over} delve deep into the underlying causes of over-squashing, examining various factors such as the graph topology and neural network architecture, particularly the depth and width. The severity of the over-squashing of a graph is closely related to the graph topology. A method to assess the degree of over-squashing in a GNN involves analyzing the number of non-zero elements within the convolution matrix. A further approach to examining over-squashing involves a comprehensive sensitivity analysis conducted on the \textit{Jacobian} of node features, as described in \cite{topping2022understanding}, which reveals that negatively curved edges contribute significantly to this over-squashing issue.

The stability problem is another prevalent issue in GNNs, where the presence of noisy or perturbed graph signals can adversely affect the performance of GNN models~\citep{jiang2023local}. Specifically, summation-based MPNNs fail to demonstrate satisfactory stability and transferability~\citep{yehudai2021local}. The process of feature aggregation in graph representation learning amplifies the negative impact of irregular entities on their neighborhoods and accumulates noisy signals during multi-layer propagation, ultimately leading to misleading embeddings. In fact, the generalization bound associated with this problem is directly proportional to the largest eigenvalue of the graph Laplacian~\citep{verma2019stability}. Therefore, employing spectral graph filters that are robust to certain structural perturbations becomes a necessary condition for achieving transferability in graph representation learning.

In this paper, our objective is to overcome the above obstacles by establishing a message passing framework operating within the spectral domain. Utilizing wavelet scattering, we can effectively capture long-range patterns in graphs, thereby mitigating the over-smoothing issue. Additionally, we draw parallels to the scattering transform to devise a layered message passing architecture. This neural network performs message passing on graphs with denser connectivity, addressing the problem of over-squashing. The meticulously designed DSMP neural network inherits the stability imparted by scattering wavelet transforms, which conserve energy. This approach empowers us to effectively capture and manipulate complex graph signal patterns, advancing the field of GNNs.

\begin{figure}[h]
    \centering
    \includegraphics[width=\textwidth]{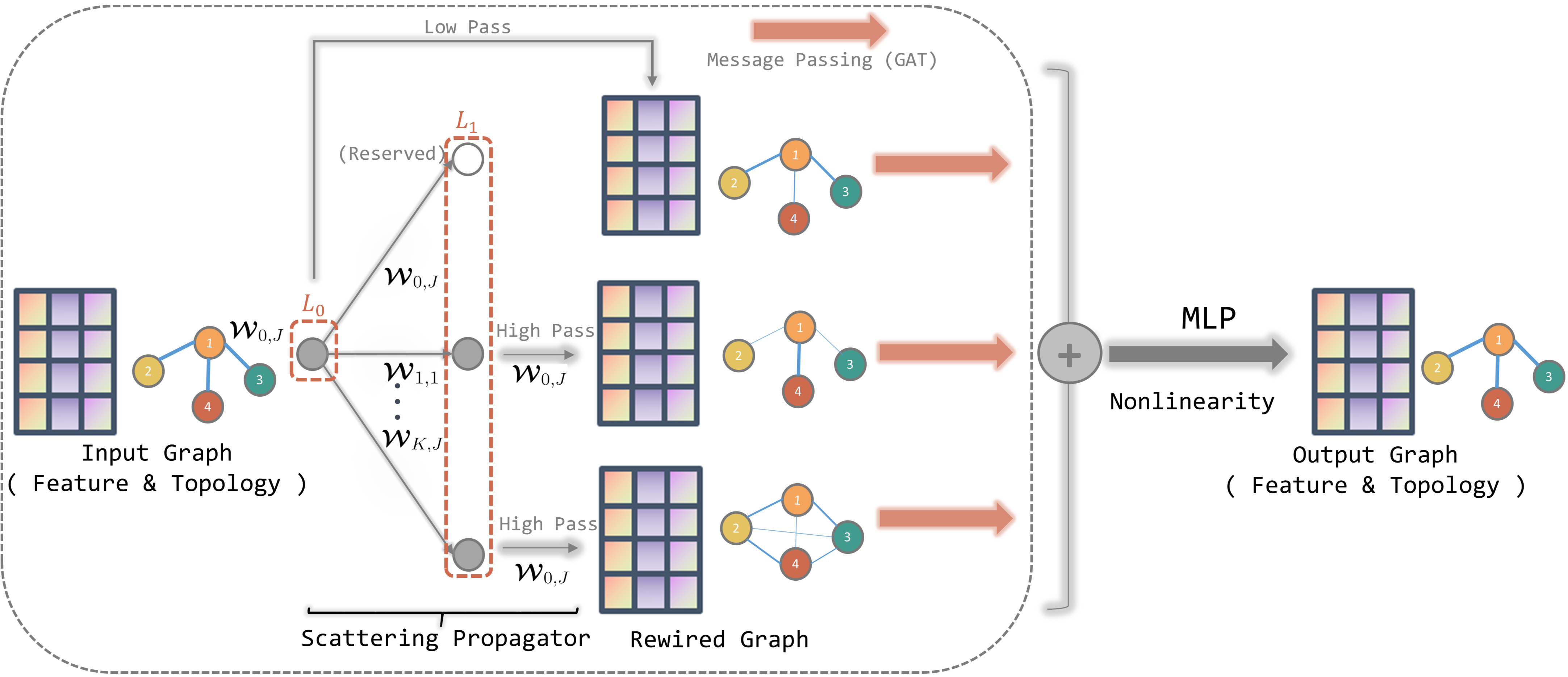}
    \caption{Architecture of the proposed scattering message passing (SMP) for graph signal embedding. In the initial scattering layer $L_0$, only the low pass coefficient is computed with framelet operator $\boldsymbol{\gW}_{0,J}$. In the following scattering layer $L_1$, the high-pass coefficients are computed with framelet filter $\boldsymbol{\gW}_{r,l}$ and cascaded with low pass filter $\boldsymbol{\gW}_{0,J}$ according to the scheme of scattering transformation. The low pass coefficient of the second scattering layer is reserved for higher-level scattering transformation. The coefficients are further processed with the MPNN which is taken to be the GAT model in this paper. After the aggregation operation, the multi-layer perceptrons (MLPs) with pointwise nonlinear activation functions are performed to increase the expressive power of the SMP model.  }
    \label{fig:smp_architecture}
\end{figure}

\section{Preliminaries}
In this section, we introduce some basic notation and properties on graphs, framelets, and scattering transform, which are based on the works of \cite{chung1997spectral,dong2017sparse,wang2019tight,bruna2013invariant,zou2020graph}. In accordance with Table \ref{app:notations} in the appendix, we give a list of notations used in this paper.

\subsection{Graphs, Function Space and Spectrum}
\label{sec:prelim1}
We consider an undirected attributed graph $\mathcal{G}=(\mathcal{V}, \mathcal{E}, \boldsymbol{X})$ with a non-empty finite set $\mathcal{V}$ of vertices, a set $\mathcal{E} \subseteq \mathcal{V} \times \mathcal{V}$ of edges between vertices in $\mathcal{V}$, and the node attributes $\boldsymbol{X} \in \mathbb{R}^{N \times d}$ encoding the feature dimensionality of each node, where $d$ denotes the dimension of features. $\boldsymbol{X}$ also denotes a feature matrix of graph signal. 
For an undirected graph $\mathcal{G}$, we denote $N:=|\mathcal{V}|$ and $M:=|\mathcal{E}|$ the numbers of vertices and edges. An edge $e \in \mathcal{E}$ with vertices $u, v \in \mathcal{V}$ is an unordered pair denoted by $(u, v)$ or $(v, u)$. The connectivity between $u$ and $v$ is denoted as $u \sim v$. We build the adjacency matrix $\mathcal{A}\in \mathbb{R}^{N \times N}$ of $\mathcal{G}$, and $\mathcal{A}$ is a mapping from $\mathcal{E}$ to $\mathcal{V} \times \mathcal{V}$ by letting 
$$
\mathcal{A}_{uv}=\left\{\begin{array}{cl}
1, &  u \sim v \\
0, & \text {otherwise}.
\end{array}\right.
$$
Note that for an undirected graph, the adjacency matrix $\mathcal{A}$ is symmetric in the sense that $\mathcal{A}(u, v)=\mathcal{A}(v, u), \forall u, v \in \mathcal{V}$. We denote the degree of a vertex $v \in \mathcal{V}$ by $\boldsymbol{d}(v):=\sum_{p \in \mathcal{V}} \mathcal{A}(v, p)$. The volume of the $\operatorname{graph} \operatorname{vol}(\mathcal{G}):=\operatorname{vol}(\mathcal{V})=\sum_{v \in \mathcal{V}} \boldsymbol{d}(v)=2M$, which is the sum of degrees of all vertices of $\mathcal{G}$. Given a subset $V_0$ of $\mathcal{V}$, the volume of $V_0$ is the sum of degrees of all nodes in $V_0$, that is, $\operatorname{vol}\left(\mathcal{V}_0\right):=\sum_{v \in \mathcal{V}_0} d(v)$.

Let $l_2(\mathcal{G}):=l_2\left(\mathcal{G},\langle\cdot, \cdot\rangle_{\mathcal{G}}\right)$ be the Hilbert space of vectors $\boldsymbol{f}: \mathcal{V} \rightarrow \mathbb{C}$ on the graph $\mathcal{G}$ equipped with the inner product
$$
\langle\boldsymbol{f}, \boldsymbol{g}\rangle_{\mathcal{G}}:=\sum_{v \in \mathcal{V}} \boldsymbol{f}(v) \overline{\boldsymbol{g}(v)}, \quad \boldsymbol{f}, \boldsymbol{g} \in l_2(\mathcal{G}),
$$
where $\bar{\boldsymbol{g}}$ is the complex conjugate to $\boldsymbol{g}$. The induced norm $\|\cdot\|_{\mathcal{G}}$ is then given by $\|\boldsymbol{f}\|_{\mathcal{G}}:=$ $\sqrt{\langle \boldsymbol{f}, \boldsymbol{f}\rangle_{\mathcal{G}}}$ for $\boldsymbol{f} \in l_2(\mathcal{G})$. For simplicity, we shall drop the subscript $\mathcal{G}$, and use $\langle\cdot, \cdot\rangle$ and $\|\cdot\|$ if no confusion arises.

Let $\delta_{\ell, \ell^{\prime}}$ be the Kronecker delta function satisfying $\delta_{\ell, \ell^{\prime}}=1$ if $\ell=\ell^{\prime}$, and $\delta_{\ell, \ell^{\prime}}=0$ if $\ell \neq \ell^{\prime}$. A finite subset $\left\{\boldsymbol{u}_{\ell}\right\}_{\ell=1}^N$ of $l_2(\mathcal{G})$ is said an orthonormal basis on $l_2(\mathcal{G})$ if
$
\left\langle\boldsymbol{u}_{\ell}, \boldsymbol{u}_{\ell^{\prime}}\right\rangle=\delta_{\ell, \ell^{\prime}}, 1 \leq \ell, \ell^{\prime} \leq N .
$
Let $\left\{\boldsymbol{u}_{\ell}\right\}_{\ell=1}^N$ be an orthonormal basis on $l_2(\mathcal{G})$. For $\ell=1, \ldots, N$, let
$
\widehat{\boldsymbol{f}}_{\ell}:=\left\langle \boldsymbol{f}, u_{\ell}\right\rangle
$
be the (generalized) Fourier coefficient of degree $\ell$ for $\boldsymbol{f} \in l_2(\mathcal{G})$ with respect to $\boldsymbol{u}_{\ell}$. 

Let $\widehat{\boldsymbol{f}}_{\ell}:=\left(\widehat{\boldsymbol{f}}_1, \ldots, \widehat{\boldsymbol{f}}_N\right) \in \mathbb{C}^N$ be the sequence of the Fourier coefficients for $\boldsymbol{f}$. Then, $\boldsymbol{f} = $ $\sum_{\ell=1}^N \widehat{\boldsymbol{f}}_{\ell} \boldsymbol{u}_{\ell}, \forall \boldsymbol{f} \in l_2(\mathcal{G})$, and Parseval's identity holds: $\|\boldsymbol{f}\|^2=\sum_{\ell=1}^N\left|\widehat{\boldsymbol{f}}_{\ell}\right|^2, \forall \boldsymbol{f} \in l_2(\mathcal{G})$. We say $\left\{\left(\boldsymbol{u}_{\ell}, \lambda_{\ell}\right)\right\}_{\ell=1}^N$ is an orthonormal eigen-pair for $l_2(\mathcal{G})$ if $\left\{\boldsymbol{u}_{\ell}\right\}_{\ell=1}^N$ is an orthonormal basis for $l_2(\mathcal{G})$ with $u_1 \equiv 1 / \sqrt{N}$ and $\left\{\lambda_{\ell}\right\}_{\ell=1}^N \subseteq \mathbb{R}$ is a nondecreasing sequence of nonnegative numbers satisfying $0=\lambda_1 \leq \cdots \leq \lambda_N$. A typical example is the eigen-pairs, that is, the set of all pairs of the eigenvectors and eigenvalues of the graph Laplacian on $\mathcal{G}$. 

We define the Laplacian for graphs $\mathcal{G}$ without loops and multiple edges. We first consider the unnormalized Laplacian matrix $L\in \mathbb{R}^{N \times N}$, also known as the combinatorial graph Laplacian, which is defined as follows:
$$
L_{uv}=\left\{\begin{array}{cl}
d_v, & u=v \\
-1, &  u \sim  v  \\
0, & \text{otherwise},
\end{array}\right.
$$
and let $\mathcal{D}$ the diagonal matrix with the $(v, v)$-th entry having value $d_v$. The Laplacian of $\mathcal{G}$ is defined to be the matrix
$$
\mathcal{L}_{uv}=\left\{\begin{array}{cl}
1, &  u=v, d_v \neq 0 \\
-\frac{1}{\sqrt{d_u d_v}}, &  u \sim v  \\
0, & \text{otherwise}.
\end{array}\right.
$$
We can write
$$
\mathcal{L}=\mathcal{D}^{-1 / 2} L \mathcal{D}^{-1 / 2},
$$
with the convention $\mathcal{D}^{-1}(v, v)=0$ for $d_v=0$. Since $\mathcal{L}$ is symmetric, its eigenvalues are all real and non-negative. The set of the eigenvalues is usually called
the spectrum of $\mathcal{L}$. The spectrum of a graph is closely related to its connectivity, and a larger value of the eigenvalue often indicates denser connectivity. For a graph $\mathcal{G}$ with $n$ vertices, there are some basic facts about the spectrum of a graph \citep{chung1997spectral}:
\begin{enumerate}[(i)] 
    \item The sum of the eigenvalues 
    $$
\sum_{i=1}^{N} \lambda_i \leq N
$$
with equality holding if and only if $\mathcal{G}$ has no isolated vertices.

\item If $\mathcal{G}$ is connected, then $\lambda_2>0$. If $\lambda_i=0$ and $\lambda_{i+1} \neq 0$, then $\mathcal{G}$ has exactly $i+1$ connected components.
\item  For all $i \leq N$, we have
$$
\lambda_i \leq 2
$$
with $\lambda_{N}=2$ if and only if a connected component of $\mathcal{G}$ is bipartite and nontrivial.
\item  The spectrum of a graph is the union of the spectra of its connected components.
\end{enumerate}
\subsection{Graph Framelet System}
\label{sec:graph_framelet_system}
In this section, we will present a concise introduction to the framelet system, tailored specifically for graph signals. Furthermore, we shall delve into the fast approximation method, focusing on the Chebyshev approximation for framelet decomposition and reconstruction operators. This method holds paramount importance in attaining computationally efficient outcomes.

A framelet, also called wavelet tight frame~\citep{dong2017sparse,wang2019tight,wang2020tight},  comprises two basic components: a \emph{filter bank} $\boldsymbol{\eta}:=\{a;b^{(1)},\dots,b^{(K)}\}$ and a set of \emph{scaling functions} $\Psi=\{\alpha;\beta^{(1)},\dots,\beta^{(K)}\}$. In this context, $a$ represents the low-pass filter, while $b^{(r)}$, where $r=1,\dots,K$, represents the $r$-th high-pass filter. These sets of filters extract both the approximate and detailed information of the input graph signal within the transformed domain, commonly referred to as the framelet domain.

Different choices of filter masks lead to distinct tight framelet systems~\citep{dong2017sparse}. In general, the selected filter bank and its associated scaling function must satisfy the following relationship:
\begin{equation} \label{eq:relations}
\widehat{\alpha}(2\xi)=\widehat{a}(\xi) \widehat{\alpha}(\xi),\quad
\widehat{\beta^{(r)}}(2 \xi)=\widehat{b^{(r)}}(\xi) \widehat{\alpha}(\xi),
\end{equation}
where $r=1, \ldots, K$ and $\xi \in \mathbb{R}$.
We give two typical examples of filters and scaling functions, as follows. 

The first one is the \emph{Haar-type} filters with one high pass, i.e., $K=1$. For $\xi\in\R$, it defines
\begin{equation*}
    \widehat{a}(\xi) = \cos(\xi/2)\mbox{~~and~~} \widehat{b^{(1)}}(\xi) = \sin(\xi/2),
\end{equation*}
with scaling functions 
\begin{equation*}
\widehat{\alpha}(\xi)=\frac{\sin (\xi / 2)}{\xi / 2}, \quad \widehat{\beta^{(1)}}(\xi)=\sqrt{1-\left(\frac{\sin (\xi / 2)}{\xi / 2}\right)^2} .
\end{equation*}
Another example of filters and scaling functions with two high passes is \citep{daubechies1992ten}:

$$
\begin{gathered}
\widehat{a}(\xi):= \begin{cases}1, & |\xi|<\frac{1}{8}, \\
\cos \left(\frac{\pi}{2} \nu(8|\xi|-1)\right), & \frac{1}{8} \leq|\xi| \leq \frac{1}{4}, \\
0, & \frac{1}{4}<|\xi| \leq \frac{1}{2},\end{cases} \\
\widehat{b^{(1)}}(\xi):= \begin{cases}0, & |\xi|<\frac{1}{8}, \\
\sin \left(\frac{\pi}{2} \nu(8|\xi|-1)\right), & \frac{1}{8} \leq|\xi| \leq \frac{1}{4}, \\
\cos \left(\frac{\pi}{2} \nu(4|\xi|-1)\right), & \frac{1}{4}<|\xi| \leq \frac{1}{2},\end{cases} \\
\widehat{b^{(2)}}(\xi):= \begin{cases}0, & |\xi|<\frac{1}{4}, \\
\sin \left(\frac{\pi}{2} \nu(4|\xi|-1)\right), & \frac{1}{4} \leq|\xi| \leq \frac{1}{2},\end{cases}
\end{gathered}
$$
where

$$
\nu(t):=t^4\left(35-84 t+70 t^2-20 t^3\right), \quad t \in \mathbb{R}
$$

The associated framelet generators $\Psi=\left\{\alpha ; \beta^1, \beta^2\right\}$ are defined by

$$
\begin{aligned}
& \widehat{\alpha}(\xi)= \begin{cases}1, & |\xi|<\frac{1}{4}, \\
\cos \left(\frac{\pi}{2} \nu(4|\xi|-1)\right), & \frac{1}{4} \leq|\xi| \leq \frac{1}{2}, \\
0, & \text { else },\end{cases} \\
& \widehat{\beta^1}(\xi)= \begin{cases}\sin \left(\frac{\pi}{2} \nu(4|\xi|-1)\right), & \frac{1}{4} \leq|\xi|<\frac{1}{2}, \\
\cos ^2\left(\frac{\pi}{2} \nu(2|\xi|-1)\right), & \frac{1}{2} \leq|\xi| \leq 1, \\
0, & \text { else, }\end{cases} \\
& \widehat{\beta^2}(\xi)= \begin{cases}0, & |\xi|<\frac{1}{2}, \\
\cos \left(\frac{\pi}{2} \nu(2|\xi|-1)\right) \sin \left(\frac{\pi}{2} \nu(2|\xi|-1)\right), & \frac{1}{2} \leq|\xi| \leq 1, \\
0, & \text { else, }\end{cases}
\end{aligned}
$$

For both practical implementation and computational efficiency, we consider the \emph{Haar-type} filters with a single high-pass filter.
By utilizing the defined filter bank and scaling functions, we can establish the undecimated framelet basis for a graph signal with the eigenpairs $\{(\lambda_\ell,\vu_\ell)\}_{\ell=1}^{N}$ of the graph Laplacian $L$. For each level $l=1,\dots,J$, the undecimated framelets corresponding to node $p$ are defined as follows:
\begin{equation}\label{eq:phi_psi}
  \begin{array}{l}
  \boldsymbol{\varphi}_{l,p}(v):=\sum_{\ell=1}^{N} \widehat{\alpha}\left(\frac{\lambda_{\ell}}{2^{l}}\right) \overline{\vu_{\ell}(p)} \vu_{\ell}(v), v\in\gV\\[2mm]
  \boldsymbol{\psi}_{l,p}^{(r)}(v):=\sum_{\ell=1}^{N} \widehat{\beta^{(r)}}\left(\frac{\lambda_{\ell}}{2^{l}}\right) \overline{\vu_{\ell}(p)} \vu_{\ell}(v), v\in\gV, \quad r=1, \ldots, K.
  \end{array}
\end{equation}
Here, $\boldsymbol{\varphi}_{l,p}(v)$ and $\boldsymbol{\psi}_{l,p}^{(r)}(v)$ represent the low-pass and the $r$-th high-pass framelet basis, respectively. 
\paragraph{Framelet Decomposition}
Next, we introduce the \emph{framelet decomposition operator} denoted as $\boldsymbol{\gW}$ for the tight framelet system~\citep{wang2019tight,wang2020tight,zheng2022decimated}, which is comprised of individual operators ${\boldsymbol{\gW}_{k,l}}\in \mathbb{R}^{N \times N}$ which we will specify shortly. In Section \ref{sec:scattering_on_graph}, we will utilize this operator to construct the SMP for the graph signal.

The framelet decomposition operator $\boldsymbol{\gW}_{k,l}$ consists of orthonormal bases for $(k,l) \in \{(0, J)\}\cup\{(1,1), \dots,(1, J), \dots(K, 1), \dots,(K, J)\}$. It transforms a given graph signal $\mX$ into a set of \emph{framelet coefficients} at multiple scales and levels. These coefficients represent the spectral representation of the graph signal $\mX$ in the transformed domain. Specifically, $\boldsymbol{\gW}_{0,J}$ includes $\boldsymbol{\varphi}_{J,p}$, where $p\in\gV$, and constructs the low-pass framelet coefficients $\boldsymbol{\gW}_{0, J}\mX$, preserving the approximate information in $\mX$. These coefficients provide a smooth representation that captures the global trend of $\mX$. On the other hand, the high-pass coefficients $\boldsymbol{\gW}_{r,l}\mX$, where $\boldsymbol{\gW}_{r,l}={\boldsymbol{\psi}_{l,p}^{(r)}, p\in\gV}$, capture detailed information at scale $r$ and level $l$. They reveal localized patterns or noise in the signal. A larger scale $k$ contains more localized information with smaller energy.

The framelet coefficients can be directly projected with $\langle\boldsymbol{\varphi}_{l,p},\mX\rangle$ and $\langle\boldsymbol{\psi}_{l,p}^{(r)},\mX\rangle$ for node $p$ at scale level $l$. For instance, if we perform eigendecomposition on a normalized graph Laplacian $\mathcal{L}$, resulting in $\mU=[\vu_1,\dots,\vu_N]\in\R^{N\times N}$ as the eigenvectors and $\Lambda=\operatorname{diag}(\lambda_1,\dots,\lambda_N)$ as the eigenvalues, we can define the corresponding filtered diagonal matrices with low-pass and high-pass filters as follows:
\begin{equation*}
\begin{aligned}
\widehat{\alpha}\left(\frac{\Lambda}{2}\right)&=\operatorname{diag}\left(\widehat{\alpha}\left(\frac{\lambda_1}{2}\right),\dots,\widehat{\alpha}\left(\frac{\lambda_n}{2}\right)\right), \
\widehat{\beta^{(r)}}\left(\frac{\Lambda}{2^l}\right)&=\operatorname{diag}\left(\widehat{\beta^{(r)}}\left(\frac{\lambda_1}{2^l}\right), \dots, \widehat{\beta^{(r)}}\left(\frac{\lambda_n}{2^l}\right)\right).
\end{aligned}
\end{equation*}

The associated framelet coefficients for the low-pass and the $r$-th high-pass signals are given by:
\begin{equation} 
\label{eq:ft_initial}
\begin{aligned}
&\boldsymbol{\gW}_{0,J}\mX=\mU \widehat{\alpha}\left(\frac{\Lambda}{2}\right) \mU^{\top} \mX , \
&\boldsymbol{\gW}_{r,l}\mX=\mU \widehat{\beta^{(r)}}\left(\frac{\Lambda}{2^{l}}\right) \mU^{\top} \mX,
\quad \forall l=1,\dots,J.
\end{aligned}
\end{equation}

\paragraph{Approximate Framelet Transformation}
To achieve an efficient framelet decomposition, we employ two strategies: a recursive formulation on filter matrices and the utilization of Chebyshev-approximated eigenvectors~\citep{wang2019tight,zheng2022decimated}.

Assuming a filter bank that satisfies \eqref{eq:relations}, we can implement the decomposition recursively as follows:
\begin{equation}
\begin{aligned}
\boldsymbol{\gW}_{r,1}\mX=\mU \widehat{\beta^{(r)}}\left(2^{-R}\Lambda\right) \mU^{\top} \mX,
\end{aligned}
\end{equation}
for the first level ($l=1$), and

\begin{equation}
\begin{aligned}
\boldsymbol{\gW}_{r,l}\mX&=\mU \widehat{\beta^{(r)}}\left(2^{-R+l-1}\Lambda\right)\widehat{\alpha}\left(2^{-R+l-2}\Lambda\right) \dots \widehat{\alpha}\left(2^{-R}\Lambda\right)
\mU^{\top} \mX \
&\quad \forall l=2,\dots,J,
\end{aligned}
\end{equation}
where the real-value dilation scale $R$ satisfies $\lambda_{\max} \leq 2^R\pi$. In
this definition, it is required for the finest scale $\frac{1}{2^{K+J}}$, that guarantees $\frac{\lambda_{\ell}}{2^{K+J-\ell}} \in (0,\pi), \forall \ell=1,2,\dots,J$.

Furthermore, we utilize an $r$-order Chebyshev polynomial approximation to achieve an efficient framelet decomposition. This approach avoids the eigendecomposition of the graph Laplacian, which can be time-consuming for large graphs.
Let the filters $\alpha$ and $\beta^{(r)}$ be approximated by $\gT_0$ and $\gT_{r}$, respectively, using $r$-order approximation. The framelet decomposition operator $\boldsymbol{\gW}_{r,l}$ is then approximated as follows:
\begin{equation}
\label{eq:chebyshev_approx} 
\boldsymbol{\gW}_{r,l}^{\natural}=
    \begin{cases}
    \gT_r\left(2^{-R} \mathcal{L}\right), & l=1, \\[1mm]
    \gT_r\left(2^{-R+l-1} \mathcal{L}\right) \gT_0\left(2^{-R+l-2} \mathcal{L}\right) \dots \gT_0\left(2^{-R} \mathcal{L}\right), & l=2,\dots,J.
    \end{cases}
\end{equation}
The above formulation can be efficiently computed using the Laplacian matrix $\mathcal{L}$, and the resulting transformation matrices exhibit a close relationship with the adjacency matrix $\mathcal{A}$.

 In a spectral-based graph convolution, graph signals are transformed to a set of coefficients $\hat{\boldsymbol{X}}=\boldsymbol{\gW} \boldsymbol{X}^{\text {in }}$ in frequency channels to learn useful graph representation. With $K$ high-pass filters, $\hat{\boldsymbol{X}}=\left\{\boldsymbol{\gW}_{0,J} \boldsymbol{X} ; \boldsymbol{\gW}_{r, l} \boldsymbol{X}: r=1, \ldots, K, l=1, \ldots, J\right\}$ constitutes of a low-pass coefficient matrix $\boldsymbol{\gW}_{0,J} \boldsymbol{X}$ and $K$ high-pass coefficient matrices $\boldsymbol{\gW}_{r, l} \boldsymbol{X}$ at $J$ scale levels. The framelet coefficient at node $i$ gathers information from its neighbors in the framelet domain of the same channel. Since the original (undecimated) framelet transforms involve integrals, we employ Chebyshev polynomial approximation for $\boldsymbol{\gW}_{r, l}$ (including $\boldsymbol{\gW}_{0,J}$), and obtain the corresponding approximated framelet transformation matrices $\boldsymbol{\gW}_{r, l}^{\natural}$ as products of Chebyshev polynomials of $A$. Let $m$ denote the highest order of the Chebyshev polynomial used. The aggregation of framelet coefficients for the neighborhood of node $i$ extends up to the $m$-th multi-hop $\mathcal{N}^{m}(i)$.

\subsection{Foundational Properties of the Graph Framelet Transform}
\label{sec:framelet_properties}
The Graph Framelet Transform, constructed from a filter bank $\boldsymbol{\eta}:=\{a;b^{(1)},\dots,b^{(K)}\}$ and scaling functions $\Psi=\{\alpha;\beta^{(1)},\dots,\beta^{(K)}\}$ as defined in Section \ref{sec:graph_framelet_system}, possesses several fundamental mathematical properties. These properties underpin its effectiveness for multi-scale graph signal analysis and are crucial for understanding its integration into neural network architectures like DSMP. The profound properties of this system, particularly its tight frame property, are the mathematical bedrock for the theoretical guarantees presented in this paper, including energy conservation, stability, and the mitigation of over-smoothing.

\subsubsection{Tight Frame and Orthogonality Property}
The most fundamental characteristic of the transform is that the collection of framelets forms a tight frame (a generalization of an orthogonal basis) for the graph signal space $l_2(\mathcal{G})$. This property is ensured by the filter design and the resulting framelet definitions.

\begin{itemize}
    \item \textbf{Mathematical Formulation}: The set of framelet basis functions $\{\boldsymbol{\varphi}_{J,p}\}_{p \in \mathcal{V}} \cup \{\boldsymbol{\psi}_{l,p}^{(r)}\}_{p \in \mathcal{V}; r=1,\dots,K; l=1,\dots,J}$ constitutes a tight frame. This means for any graph signal $\mX \in l_2(\mathcal{G})$, the following energy preservation (Parseval) identity holds:
    \begin{equation}
        \|\mX\|^2 = \sum_{p \in \mathcal{V}} |\langle \boldsymbol{\varphi}_{J,p}, \mX \rangle|^2 + \sum_{r=1}^{K} \sum_{l=1}^{J} \sum_{p \in \mathcal{V}} |\langle \boldsymbol{\psi}_{l,p}^{(r)}, \mX \rangle|^2.
        \label{eq:tight_frame_identity}
    \end{equation}
    Equivalently, in the matrix formulation using the framelet decomposition operators $\boldsymbol{\gW}$, this translates to the critical operator identity:
    \begin{equation}
        \boldsymbol{\gW}_{0,J}^\top \boldsymbol{\gW}_{0,J} + \sum_{r=1}^{K}\sum_{l=1}^{J} \boldsymbol{\gW}_{r,l}^\top \boldsymbol{\gW}_{r,l} = \mathbf{I},
        \label{eq:frame_operator_identity}
    \end{equation}
    where $\mathbf{I}$ is the identity matrix. This identity is the direct source of the energy conservation proven in Proposition 1 and is indispensable for the proof of Theorem 3.

    \item \textbf{Spectral Interpretation and Filter Condition}: The operator identity \eqref{eq:frame_operator_identity} stems from an equivalent condition in the spectral domain. Given the eigendecomposition $\mathcal{L} = \mU \Lambda \mU^\top$, the condition \eqref{eq:frame_operator_identity} is satisfied if and only if the filters fulfill the following condition for all eigenvalues $\lambda_\ell$:
    \begin{equation}
        \left| \widehat{\alpha}\left( \frac{\lambda_\ell}{2^J} \right) \right|^2 + \sum_{r=1}^{K} \sum_{j=1}^{J} \left| \widehat{\beta^{(r)}}\left( \frac{\lambda_\ell}{2^j} \right) \right|^2 = 1, \quad \forall \ell.
        \label{eq:unitary_extension_principle}
    \end{equation}
    Equation \eqref{eq:unitary_extension_principle} is the \textit{Unitary Extension Principle (UEP)} condition for graphs. It guarantees that the analysis and synthesis operations performed by the framelet transforms are adjoint and inverse to each other, forming the foundation for perfect reconstruction and stability.
\end{itemize}

\subsubsection{Perfect Reconstruction and Energy Conservation}
Directly derived from the tight frame property are two cornerstone features.

\begin{itemize}
    \item \textbf{Perfect Reconstruction}: Any graph signal can be losslessly synthesized from its framelet coefficients:
    \begin{equation}
        \mX = \boldsymbol{\gW}_{0,J}^\top (\boldsymbol{\gW}_{0,J}\mX) + \sum_{r=1}^{K}\sum_{l=1}^{J} \boldsymbol{\gW}_{r,l}^\top (\boldsymbol{\gW}_{r,l}\mX).
        \label{eq:perfect_reconstruction}
    \end{equation}
    This property ensures no information loss during the transformation, making it a stable and invertible feature extractor.

    \item \textbf{Energy Conservation}: A direct consequence of \eqref{eq:tight_frame_identity} is the preservation of the signal's $\ell_2$-norm (energy):
    \begin{equation}
        \|\mX\|^2 = \|\boldsymbol{\gW}_{0,J}\mX\|^2 + \sum_{r=1}^{K}\sum_{l=1}^{J} \|\boldsymbol{\gW}_{r,l}\mX\|^2.
        \label{eq:energy_conservation_framelet}
    \end{equation}
    This is the precise mathematical statement proven in Proposition 1, where the Dirichlet energy $E_{\mathcal{G}}(\mX) = \tr(\mX^\top \mathcal{L} \mX)$ is shown to be similarly decomposed. The boundedness of Dirichlet energy across DSMP layers (Theorem 3) fundamentally relies on this conservation property.
\end{itemize}

\subsubsection{Stability (Lipschitz Continuity)}
The orthonormal nature of the tight frame (as a system) guarantees that the framelet transform itself is a non-expansive operator.

\begin{itemize}
    \item \textbf{Bounded Norm}: For any signal $\mX$, it follows from \eqref{eq:energy_conservation_framelet} that $\|\boldsymbol{\gW}_{0,J}\mX\| \leq \|\mX\|$ and $\|\boldsymbol{\gW}_{r,l}\mX\| \leq \|\mX\|$. This bound is explicitly used in the proof of Lemma 2.
    \item \textbf{Stability to Input Perturbations}: Consequently, the transform is Lipschitz continuous with constant 1. For two signals $\mX$ and $\widetilde{\mX}$,
    \begin{equation}
        \|\boldsymbol{\gW}_{0,J}\mX - \boldsymbol{\gW}_{0,J}\widetilde{\mX}\| \leq \|\mX - \widetilde{\mX}\|, \quad \|\boldsymbol{\gW}_{r,l}\mX - \boldsymbol{\gW}_{r,l}\widetilde{\mX}\| \leq \|\mX - \widetilde{\mX}\|.
        \label{eq:framelet_lipschitz}
    \end{equation}
    This inherent stability of the framelet building blocks is a key factor that propagates through the architecture to ensure the overall stability of the DSMP model, as formalized in Theorem \ref{thm:stability}.
\end{itemize}

\subsubsection{Multiscale and Spectral-Spatial Localization}
Beyond orthogonality, the transform possesses structuring properties critical for graph analysis.

\begin{itemize}
    \item \textbf{Multiscale Analysis}: By employing a cascade of filters at dyadic scales $2^{-l}$, the transform decomposes a signal into a hierarchy of components. The low-pass coefficients $\boldsymbol{\gW}_{0,J}\mX$ capture smooth, global trends, while the high-pass coefficients $\boldsymbol{\gW}_{r,l}\mX$ capture localized details at scale $l$. This allows DSMP to aggregate information from both local neighborhoods and global contexts, directly addressing the over-squashing problem.
    \item \textbf{Spectral-Spatial Localization}: Each framelet $\boldsymbol{\psi}_{l,p}^{(r)}$ is approximately localized in both the spectral (frequency) domain and the spatial (graph vertex) domain. Spectrally, it acts as a band-pass filter. Spatially, its energy is concentrated around the center node $p$ within a hop distance related to scale $l$. This dual localization enables efficient and interpretable feature extraction.
\end{itemize}

In summary, the tight frame property, along with the consequent energy conservation and stability, forms the indispensable mathematical foundation. These properties are not merely incidental but are actively leveraged in the proofs of Proposition 1 and Theorems 3 and 5 to establish the DSMP model's robustness against over-smoothing, over-squashing, and input perturbations.

\subsection{Graph Scattering Transform}
\label{sec:scattering_on_graph}
The scattering transform is a multi-scale signal analysis method that provides stable representations using framelets. The graph scattering transform extends this method to analyze graph signals. A scattering transform can be constructed by cascading three building blocks: a collection of paths containing the ordering of the scales of framelets, a pointwise modulus activation function, and a low-pass framelet operator. Following the same framework, given a graph $\mathcal{G}$ and input feature matrix $\mX\in l^2(\mathcal{G})$, we define an analogous graph scattering transformation $\Phi_{\mathcal{G}}(\mX)$ by cascading three building blocks: a collection of paths $\mathbf{p}=\left\{(1,1), \dots,(1, J), \dots,(K, J)\right\}^{m-1}$, each containing the high-pass components of framelet decomposition operators $\boldsymbol{\gW}_{r,l}, \; 1\leq r \leq K,\; 1\leq l\leq J,$ where $m$ is the depth of the transformation; a pointwise absolute value function $\hat{\sigma}$; and a low-pass component  $\boldsymbol{\gW}_{0,J}$. More specifically, the first layer of scattering transformation can be formulated as
\begin{equation}
\label{eq:smp_1st}
\phi_1(\mathcal{G}, \mX)=\boldsymbol{\gW}_{0,J}\mX,
\end{equation}
followed by the second layer 
\begin{equation}
\label{eq:smp_2nd}
\phi_2(\mathcal{G}, \mX)=\boldsymbol{\gW}_{0,J}\hat{\sigma} (\boldsymbol{\gW} \mX)=\left\{\boldsymbol{\gW}_{0,J}\sigma (\boldsymbol{\gW}_{r,l} \mX)\right\}_{{1 \leq r \leq K},{1 \leq l \leq J}},
\end{equation}
and the propagation on the $m$-th layer reads
\begin{equation}
\label{eq:smp_mth}
\phi_m(\mathcal{G}, \mX)=\boldsymbol{\gW}_{0,J}\underbrace{\hat{\sigma}( \boldsymbol{\gW} \cdots \hat{\sigma} (\boldsymbol{\gW} \mX)) }_{(m-1)*\gW}  =\left\{\boldsymbol{\gW}_{0,J} (\hat{\sigma} ( \boldsymbol{\gW}_{r,l} \cdots \hat{\sigma} (\boldsymbol{\gW}_{r,l} \mX))\right\}_{{1 \leq r \leq K},{1 \leq l \leq J}}.
\end{equation}
The representation obtained by concatenating output features from $m$ layers of such transformation is thus
\begin{equation}
\Phi_{\mathcal{G}}(\mX)=\left\{\phi_1(\mathcal{G}, \mX),\phi_2(\mathcal{G}, \mX), \ldots, \phi_{m-1}(\mathcal{G}, \mX)\right\}.
\label{eq:scatter_prop}
\end{equation}

\subsection{Graph Neural Networks and Message Passing}
\label{sec:gnn_with_os}
 Given a graph with adjacency matrix $\mathcal{A}$ and input feature matrix $\boldsymbol{X}^{\text {in }}$, a graph convolution produces a matrix representation $\boldsymbol{X}^{\text {out }}$. Let$\boldsymbol{X}_{i}^\text{in}$ denote the $i$-th row of $\boldsymbol{X}^\text{in}$, that is, the feature for node $i$. At a specific layer, the propagation for the $i$-th node reads
\begin{equation}
\label{eq:mpnn}
    \boldsymbol{X}_{i}^{\text {out }}=\Gamma\left(\boldsymbol{X}_{i}^{\text {in }}, \square_{j \in \mathcal{N}(i)} \phi\left(\boldsymbol{X}_{i}^{\text {in }}, \boldsymbol{X}_{j}^{\text {in }}, \mathcal{A}_{i j}\right)\right),
\end{equation}
where $\square(\cdot)$ is a differentiable and permutation invariant aggregation function, such as summation, average, or maximization. The set $\mathcal{N}(i)$ includes the $i$-th node and its 1-hop neighbors. Both $\Gamma(\cdot)$ and $\phi(\cdot)$ are MLPs. 
\paragraph{Over-smoothing}
The over-smoothing problem commonly exists in GNN models with message passing schemes, which makes it difficult to achieve optimal performance and to implement in real-world applications. The proposed DSMP framework is theoretically proved to solve the over-smoothing problem in the sense of \textit{Dirichlet energy}. The \textit{Dirichlet energy} quantifies the amount of variation or oscillation in graph signal values across nodes.

The \textit{Dirichlet energy} of graph $\mathcal{G}$ is defined as follows:
\begin{equation}
\label{eq:dirichlet_energy}
E_{\mathcal{G}}(\mX) = \frac{1}{2} \sum_{i=1}^{N} \sum_{j=1}^{N} \mathcal{A}_{ij} \left(\frac{\mX_i}{\sqrt{d_i}} - \frac{\mX_j}{\sqrt{d_j}}\right)^2,
\end{equation}
where $\mathcal{A}_{ij}$ represents the connection strength between nodes $i$ and $j$ in the adjacency matrix. The \textit{Dirichlet energy} measures the squared differences between connected node signal values, weighted by the adjacency matrix. A higher \textit{Dirichlet energy} indicates greater signal variation or less smoothness, while a lower value indicates smoother signal behavior.
\paragraph{Over-squashing}
According to the recent theory on over-squashing of MPNN \citep{di2023over}, the graph structure plays a pivotal role in the quality of signal propagation across nodes. 
A classic measure of "bottlenecked-ness" of a graph structure is the Cheeger constant \citep{chung1997spectral}
$$
h(\mathcal{G}):=\min _{\varnothing \neq \mathcal{S} \subset \mathcal{V}} \frac{|\partial \mathcal{S}|}{\min (\operatorname{vol} \mathcal{S}, \operatorname{vol} \overline{\mathcal{S}})},
$$
where $\operatorname{vol} \mathcal{S}=\sum_{v \in \mathcal{S}} d_v, \overline{\mathcal{S}}=\mathcal{V} \backslash \mathcal{S}$ and $\partial \mathcal{S}=\{(u, v): u \in \mathcal{S}, v \in \overline{\mathcal{S}},(u, v) \in \mathcal{E}\}$ is the edge boundary of $\mathcal{S} \subset \mathcal{V}$. This definition considers the connectivity between two complementary subsets of nodes in terms of the number of connecting edges and their volume, measured by the sum of the degrees of the nodes. A graph with two tightly connected communities ( $S$ and $\mathcal{V} \backslash \mathcal{S}$) and few
inter-community edges has a small Cheeger constant. Intuitively, a large value of $h(\mathcal{G})$ indicates that the graph has no significant bottlenecks, meaning that each part of the graph is well-connected to the rest through a significant fraction of its edges. A higher $h(\mathcal{G})$ implies more connections and less over-squashing on the graph scale.

In general, computing the exact value of $h(\mathcal{G})$ is a computationally challenging task. However, it is possible to provide lower and upper bounds for the Cheeger constant in terms of the spectral gap. As mentioned in Section \ref{sec:prelim1}, the spectral gap of a graph is defined as the difference $\lambda_2 - \lambda_1$, where $\lambda_1 = 0$.
The discrete Cheeger inequality \citep{cheeger1970lower,alon1984eigenvalues} establishes a relationship between the spectral gap of $\mathcal{G}$ and its Cheeger constant:
\begin{equation}
\label{eq:cheeger}
\frac{\lambda_2}{2} \leq h(\mathcal{G}) \leq \sqrt{2 \lambda_2}.
\end{equation}
The discrete Cheeger inequality provides us with the insight that the spectral gap is non-zero if and only if $h(\mathcal{G})$ is non-zero. Consequently, a graph $\mathcal{G}$ is considered connected if and only if $\lambda_2 > 0$. As the spectral gap can be efficiently computed, we utilize it as a measure of graph bottlenecking, serving as a proxy for the Cheeger constant. Thus, the Cheeger constant is closely related to the over-squashing issue in GNNs.

\section{Related works}
\subsection{Message Passing Neural Network}
MPNN~\citep{gilmer2017neural} has become a computational paradigm for GNN models to learn from graph-structured data. In one message-passing step, each node of the graph updates its state by aggregating messages flowing from its direct neighbors. Different GNN layers can be derived by varying the specific aggregation operation. For instance, the graph convolutional network (GCN)~\citep{kipf2017semi} adds up the degree-normalized node attributes from neighbors (including itself).
The graph attention network (GAT)~\citep{velivckovic2017graph}  and other variants~\citep{xie2020mgat,brody2021attentive} refine the aggregation weights by utilizing the attention mechanism. The expressive power of MPNNs has been verified by the Weisfeiler-Leman graph isomorphism test~\citep{xu2018powerful,morris2019weisfeiler,bodnar2021weisfeiler}, based on which the graph isomorphism network (GIN) has been proposed to become more powerful in expressivity.

\subsection{Over-smoothness and Over-squashing in MPNN}
The over-smoothing issue was initially identified by~\cite{li2018deeper}, who observed that the output embeddings of different nodes tend to converge to
similar vectors after going through multiple layers of signal
smoothing operations. This issue has been a focal point in the study of GNNs and poses a challenge to the performance of deep GNNs.
On the other hand, using only a small number of message-passing operations limits the expressivity of the model. When layers are deeply stacked, the embeddings of connected nodes tend to converge, becoming indistinguishable. Some~\citep{Sohir} generalized the concept of over-smoothing, which has traditionally been studied in undirected graphs, to directed graphs. They also extended the Dirichlet energy by introducing a novel directed symmetrically normalized Laplacian. To address the over-smoothing issue, they proposed fractional graph Laplacian Neural Ordinary Differential Equations (Neural ODEs), which offer a flexible framework for regulating signal propagation and preserving the distinctiveness of node features.

Another commonly investigated issue is the over-squashing problem. Such phenomenon has been explored in relation to graph curvature, based on which SDRF~\citep{topping2022understanding} and BORF~\citep{nguyen2023revisiting} proposed graph rewiring methods to address the over-squashing problem. FoSR~\citep{karhadkar2022fosr} proposed a rewiring strategy based on spectral expansion theory~\citep{banerjee2022oversquashing}, aiming to rewire the graph and obtain a larger spectral gap.
Different from rewiring the input graph, graph transformer models such as GraphTrans~\citep{wu2021representing}, Graphormer~\citep{ying2021transformers}, GPS~\citep{rampavsek2022recipe}, and U2GNN~\citep{nguyen2022universal} took a special approach to solve the over-squashing issue. These models implicitly learn on a fully connected graph inspired by the Transformer architecture~\citep{vaswani2017attention}.
Using the partial derivative between pairs of graph vertices, \citep{di2023over} showed different perspective to analyze the over-squashing problem from the graph topology, depth of message passing scheme, and number of neurons, based on which more efficient MPNNs could be designed to alleviate the over-squashing problem. One straightforward solution involves increasing the number of message passing layers, yet this strategy can potentially exacerbate the over-smoothing issue and amplify the risk of vanishing gradients. Alternatively, widening the networks can alleviate the over-squashing problem, but this may simultaneously introduce instability concerns.

\subsection{Spectral Methods on GNNs and Scattering Transform}
In recent years, spectral-based graph convolutions have shown promising performance in transferring trained graph convolutions across different graphs, indicating their transferability and generalizability~\citep{levie2019transferability,maskey2022generalization}. Furthermore, the stability of spectral-based methods in the face of input graph perturbations has been emphasized in various studies~\citep{gama2020stability,ruiz2021graph}.
One notable work by \citet{zheng2021framelets} extracts multi-scale features from graph data and integrates the framelet coefficients with a learnable GNN model, achieving good results in various graph-related tasks. 

Scattering methods, inspired by the layered architecture of framelet transformations, have been used to emulate deep neural networks and serve as powerful feature extractors that capture multi-scale group invariant features from data~\citep{mallat2012group}. The scattering transform has also been employed to analyze and interpret convolutional neural networks~\citep{bruna2013invariant}. Similarly to the Euclidean scattering transform, the graph scattering transform applies a set of framelet filters to graph signal\citep{gama2019stability,gao2019geometric}.

In recent years, prominent researchers including \citep{zou2020graph, zou2019encoding, gama2018diffusion, gama2019stability} have successfully extended the scattering transform to graph data. By leveraging graph framelets in both spectral and spatial domains, these extensions have generated stable and equivariant feature representations for instances and their intricate relationships. Furthermore, efforts have been made to integrate geometric information with graph scattering, leading to enhanced methods for graph analysis~\citep{gao2019geometric, chew2022manifold, xiong2023anisotropic}.

Despite the utilization of predefined framelets in scattering transforms, it is feasible to construct learnable filters within this framework. Such a new approach was initially implemented in the Euclidean domain~\citep{oyallon2017scaling, gauthier2022parametric} and more recently in the graph domain~\citep{min2020scattering, zhang2022hierarchical}. More recently, \cite{zhang2022hierarchical} employed diffusion wavelet techniques to develop a hierarchical diffusion scattering (HDS) network, seamlessly integrating it with GCN (HDS-GCN) and GAT (HDS-GAT) models. These hierarchical scattering features significantly enhance the performance of GNN models, further broadening their capabilities and scope.

\section{Deep Scattering Message Passing}
In this section, we introduce a novel graph embedding neural network called the DSMP(Deep Scattering Message Passing) model, designed to enhance the learning capability of scattering transform. The foundational component is the SMP module, which enables the scalability of the DSMP model. The main obstacle to scalability for common GNN models is the over-smoothing problem, which is theoretically proved to be circumvented by the proposed model. We also prove that the newly designed DSMP can alleviate the over-squashing of message passing. We also prove the stability of the DSMP.

\subsection{Framelet Scattering Transform}
We utilize framelets to construct the graph scattering transform, which we refer to as the \emph{framelet scattering transform}. The detailed introduction and efficient computation of framelet decomposition operators $\boldsymbol{\gW}_{r,l}$ are provided in section \ref{sec:graph_framelet_system}.


In Figure~\ref{fig:smp_architecture}, we illustrate a two-layer scattering propagation process using the framelet transformation matrices, where the pointwise activation function $\sigma$ is moved outside. These framelet transformations provide a weighted aggregation scheme for graph features, which can be regarded as weighted adjacency matrices for message passing and also rewired graph topology as the level of framelet transformation $\ell \geq 2$. The low pass coefficient is directly computed with framelet operator $\boldsymbol{\gW}_{0,J}$ in the initial scattering layer. When the scattering layer increases, the low pass coefficient of the following scattering layer is reserved for higher-level scattering transformation. 

\subsection{Deep Scattering Message Passing Model}
The scattering transform is a multi-scale feature extraction method without learnable parameters, which entangles its learning ability in a data-driven paradigm. Using message passing scheme, we propose the SMP framework to enable scattering transform learning ability.
The DSMP can be constructed with multi-layer SMP modules. For the update of graph node feature $\boldsymbol{X}^{(t)}\in \mathbb{R}^{N \times d}$ at layer $t$, we define the following diffusion propagation process:
\begin{equation}
    \boldsymbol{X}^{(t)}=\boldsymbol{X}^{(t-1)}+\boldsymbol{Z}^{(t)}.
\end{equation}
Below, we define one-step propagation of SMP with a depth of 2:
\begin{equation}
    \begin{aligned}
        \boldsymbol{X}_{j}^{h_{r,l}} &= \sum_{p \in \mathcal{N}(j)}\boldsymbol{\gW}_{r,l}^{p,j} \boldsymbol{X}_p^{(t-1)},\\
        \boldsymbol{Z}_{i}^{(t)} &= \Gamma\left(\sum_{j \in \mathcal{N}(i)}  \boldsymbol{\gW}_{0,J}^{i,j} \boldsymbol{X}_{j}^{(t-1)}\boldsymbol{\theta}_0+  \sum_{r=1}^K\sum_{l=1}^J\boldsymbol{\gW}_{0,J}^{i,j}\boldsymbol{X}_{j}^{h_{r,l}}\boldsymbol{\theta}_{r,l} \right),
    \end{aligned}
    \label{eq:smpnn}
\end{equation}
where the first equation denotes the high-pass coefficients of node $j$ extracted by transformation matrix $\boldsymbol{\gW}_{r,l}\in \mathbb{R}^{N \times N}$. The message passing process specifically aggregates neighboring nodes from $\mathcal{N}(j)$, where the adjacency is determined by the non-zero terms of $\boldsymbol{\gW}_{r,l}$. In the case of the framelet transformation matrix $\boldsymbol{\gW}_{r,l}$, if the term in the $p$-th row and $j$-th column is non-zero, denoted as $\boldsymbol{\gW}_{r,l}^{p,j}\neq 0$, it will influence the aggregation process. To adjust for different embeddings, $\boldsymbol{\theta}_0\in \mathbb{R}^{d \times h}$ and $\boldsymbol{\theta}_{r,l}\in \mathbb{R}^{d \times h}$ are introduced as learnable weighting parameters with hidden dimension size of $h$.

The propagation scheme described above adheres to the fundamental principle of feature aggregation in message passing. For simplicity on notation, we give the matrix formulation as
\begin{equation}
    \begin{aligned}
        &\boldsymbol{X}^{(t)}=\boldsymbol{X}^{(t-1)}+\boldsymbol{Z}^{(t)}, \\
        &\boldsymbol{Z}^{(t)}=\Gamma\left( \boldsymbol{\gW}_{0,J}\boldsymbol{X}^{(t-1)}\boldsymbol{\theta}_0+\sum_{r=1}^K\sum_{l=1}^J\boldsymbol{\gW}_{0,J}\boldsymbol{\gW}_{r,l}\boldsymbol{X}^{(t-1)}\boldsymbol{\theta}_{r,l} \right).
    \end{aligned}
    \label{eq:matrix_form}
\end{equation}
As shown in \eqref{eq:matrix_form}, the first term corresponds to low pass framelet coefficients weighted by the learnable parameter $\boldsymbol{\theta}_0$ from the GAT model. The other part consists of high-pass framelet coefficients multiplied with low-pass transformation matrix $\boldsymbol{\gW}_{0,J}$, which is designed following the architecture of two layers of scattering transform. 

Following aggregation and weighting in the spectral domain, the MLPs are applied to increase the expressive power of the SMP model. For the sake of notational simplicity, we represent the MLPs, denoted as $\Gamma$, with the parameter matrix $\mathbf{P}$ and nonlinear activation function ReLU~\citep{krizhevsky2017imagenet} $\sigma(\cdot)$. Specifically, for an input signal $\boldsymbol{X}$, $\Gamma (\boldsymbol{X})=\sigma(\mathbf{P}\boldsymbol{X})$. Worth noting is that the enhancement of the traditional pointwise modulus activation function is used in scattering transformations. This improvement has been achieved by incorporating the pointwise ReLU function within the MLPs. Importantly, this modification maintains and indeed enhances the integrity of the underlying scattering transformation design. Specifically, the ReLU function fulfills a role analogous to $\sigma(\cdot)$, contributing to the model's nonlinear approximation power. As illustrated in Figure~\ref{fig:dsmp_architecture}, the DSMP model is constructed by cascading multiple SMP modules to expand the scale of this deep learning model.
\begin{figure}
    \centering
    \includegraphics[width=\textwidth]{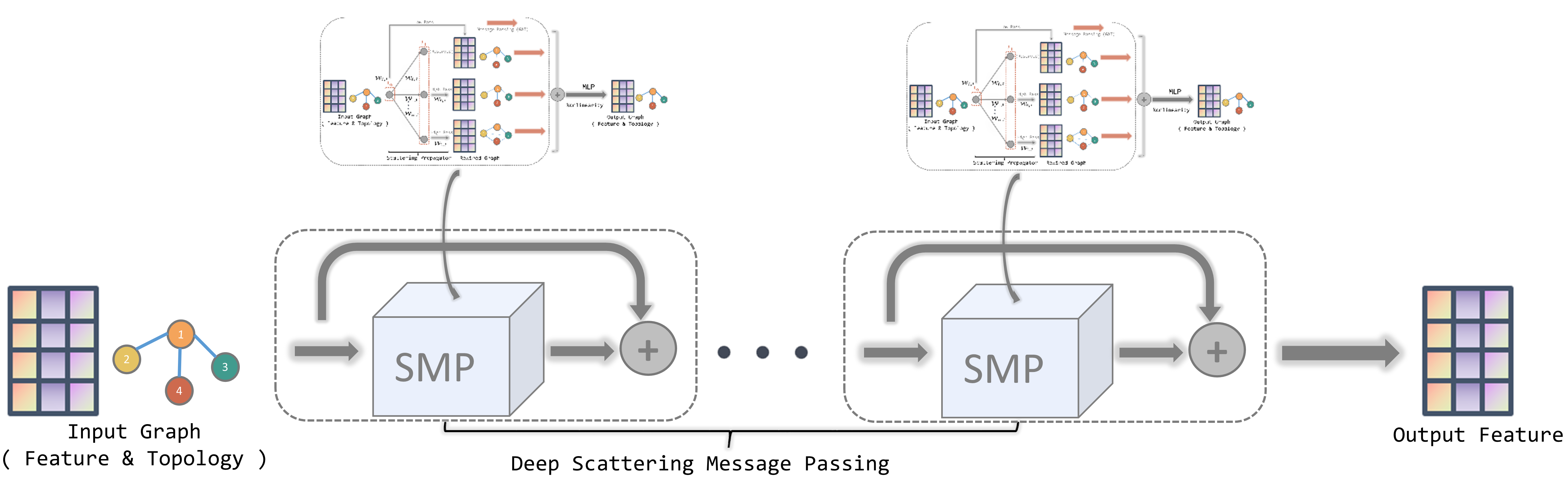}
    \caption{Overall architecture of the DSMP scheme. Following the propagation process of \eqref{eq:matrix_form}, the DSMP is constructed by cascading multiple SMP modules.}
    \label{fig:dsmp_architecture}
\end{figure}

\subsection{Limiting Over-smoothing in DSMP}
\label{sec:over-smoothing}
We utilize the \textit{Dirichlet energy} as a measure of the level of over-smoothness in message passing models. Firstly, we establish the energy conservation property of the framelet transformation. Subsequently, we provide proof demonstrating how the DSMP model effectively constrains the graph embedding.

\begin{proposition}[Energy conservation] Let $\mX$ be a feature matrix from the graph signal $\mathcal{G}$ and $\left\{\boldsymbol{\gW}_{0,0}; \boldsymbol{\gW}_{r, l}: r=1, \ldots, K, l=1, \ldots, J\right\}$ be the framelet transformation matrix. Then, 
    \begin{equation}\label{eq:DE preserve}
    E_{\mathcal{G}}(\mX) = E_{\mathcal{G}}(\boldsymbol{\gW}_{0,J}\mX)+\sum_{r=1}^{K}\sum_{l=1}^{J} E_{\mathcal{G}}(\boldsymbol{\gW}_{r,l}\mX),
\end{equation}
where framelet transformation breaks the total energy $E_{\mathcal{G}}(\mX)$ into low pass energy $E_{\mathcal{G}}(\boldsymbol{\gW}_{0,J}\mX)$ and high-pass energy $E_{\mathcal{G}}(\boldsymbol{\gW}_{r,l}\mX)$.
\end{proposition}

\begin{proof}
By the definition of \textit{Dirichlet energy}, we have $E_{\mathcal{G}}(\mX) = \tr(\mX^\top\mathcal{L}\mX)$ and  $\mathcal{L} = \mU \Lambda \mU^{\top}$. Then, according to \eqref{eq:ft_initial}, we have   
\begin{equation*}  
    E_{\mathcal{G}}(\boldsymbol{\gW}_{0,J}\mX) = \tr\left(\mX^{\top}\boldsymbol{\gW}_{0,J}^{\top}\mathcal{L}\boldsymbol{\gW}_{0,J}\mX\right)= \tr\left(\mX^{\top}\mU \widehat{\alpha}\left(\frac{\Lambda}{2}\right)^2 \Lambda \mU^{\top} \mX \right) ,
\end{equation*}
and $\forall l=1,\dots,J, $ 
\begin{equation}
    E_{\mathcal{G}}(\boldsymbol{\gW}_{r,l}\mX) = \tr\left(\mX^{\top}\boldsymbol{\gW}_{0,J}^{\top}\mathcal{L}\boldsymbol{\gW}_{r,l}\mX\right)= \tr\left(\mX^{\top}\mU \widehat{\beta^{(r)}}\left(\frac{\Lambda}{2^{l}}\right)^2 \Lambda \mU^{\top} \mX \right).
    \label{eqn:energy}
\end{equation}
Since $\Lambda=\operatorname{diag}(\lambda_1,\dots,\lambda_N)$ and the tightness of Framelet system requires the functions in $\Psi$ satisfy
\begin{align}
   1 = \left|\FT{\scala}\left(\frac{\eigvm}{2^{J}}\right)\right|^{2} + \sum_{r=1}^{K}\left|\FT{\scalb^{(r)}}\left(\frac{\eigvm}{2^{J}}\right)\right|^{2} &\quad \forall
  \ell=1,\ldots,\NV, \label{thmeq:nrm:alpha:beta}\\
     \left|\FT{\scala}\left(\frac{\eigvm}{2^{l}}\right)\right|^{2}
    = \left|\FT{\scala}\left(\frac{\eigvm}{2^{l}}\right)\right|^{2} + \sum_{r=1}^{K}\left|\FT{\scalb^{(r)}}\left(\frac{\eigvm}{2^{l}}\right)\right|^{2} &\quad \forall
 \begin{array}{l}
 \ell=1,\ldots,\NV,\\
 l=1,\ldots,J-1,
 \end{array}\label{thmeq:2scale:alpha:beta}
 \end{align}
we thus have 
\begin{equation}
    \widehat{\alpha}\left(\frac{\Lambda}{2}\right)^2 + \sum_{r=1}^K\sum_{l=1}^J \widehat{\beta^{(r)}}\left(\frac{\Lambda}{2^{l}}\right)^2 = \mI.
    \label{eqn:partition2}
\end{equation}
Therefore, combining \eqref{eqn:energy} and \eqref{eqn:partition2}, we obtain that
\begin{equation*}
\begin{aligned}
    & E_{\mathcal{G}}(\boldsymbol{\gW}_{0,J}\mX)+\sum_{r=1}^{K}\sum_{l=1}^{J} E_{\mathcal{G}}(\boldsymbol{\gW}_{r,l}\mX) \\
    =& \tr\left(\mX^{\top}\mU \widehat{\alpha}\left(\frac{\Lambda}{2}\right)^2 \Lambda \mU^{\top} \mX \right) +  \sum_{r=1}^{K}\sum_{l=1}^{J} \tr\left(\mX^{\top}\mU \widehat{\beta^{(r)}}\left(\frac{\Lambda}{2^{l}}\right)^2 \Lambda \mU^{\top} \mX \right) \\
    =& \tr\left(\mX^{\top}\mU \left( \widehat{\alpha}\left(\frac{\Lambda}{2} \right)^2 + \sum_{r=1}^{K}\sum_{l=1}^{J}\widehat{\beta^{(r)}}\left(\frac{\Lambda}{2^{l}}\right)^2 \right) \Lambda \mU^{\top} \mX \right) \\
    =& \tr\left(\mX^{\top}\mU\Lambda\mU^{\top}\mX \right) \\
    =& E_{\mathcal{G}}(\mX),
\end{aligned}
\end{equation*}
thus completing the proof.
\end{proof}


The following theorem shows that the \textit{Dirichlet energy} of the feature at the $t$-th layer $\boldsymbol{X}^{(t)}$ is not less than that of the $(t-1)$-th layer $\boldsymbol{X}^{(t-1)}$ in DSMP model. In fact, $\boldsymbol{X}^{(t)}$ and $\boldsymbol{X}^{(t-1)}$ are equivalent up to some constant, and the constant with respect to the upper bound depends jointly on the level of the framelet decomposition and the learnable parameters' bound.

 Although the asymptotic behavior of the output $X^{(L)}$ of the GCN as $L \rightarrow \infty$ is investigated with over-smoothing property, to facilitate the analysis for our model MPNN, we also sketch the proof for completeness here.
\begin{lemma} 
\label{lemma:activation}
\citep{oono2019graph}
For any $\boldsymbol{X} \in \mathbb{R}^{n \times d}$ and activation function ReLU $\sigma(\cdot)$, we have
\begin{itemize}
\item $E_{\mathcal{G}}\left(\boldsymbol{X}\boldsymbol{\Theta} \right) \leq \mu_{\max }^2 E_{\mathcal{G}}(\boldsymbol{X}),$ where $\mu_{\max }^2$ denotes the square of maximum singular value of weight matrix $\boldsymbol{\Theta}\in \mathbb{R}^{d \times h}$.\\
\item $E_{\mathcal{G}}(\sigma(\boldsymbol{X})) \leq E_{\mathcal{G}}(\boldsymbol{X})$.
\end{itemize}
\end{lemma}

\begin{theorem}[Dirichlet Energy] 
Let us consider the matrix formulation of DSMP as in \eqref{eq:matrix_form}. We can bound the \textit{Dirichlet energy} of the graph node feature $\boldsymbol{X}^{(t)}$ at layer $t$ as follows:
$$
E_{\mathcal{G}}\left(\boldsymbol{X}^{(t-1)}\right) \leq E_{\mathcal{G}}\left(\boldsymbol{X}^{(t)}\right) \leq \Bigl(1+\frac{\sqrt{5}}{2}\mathbf{C} |\mu_{\max}|\Bigr)^2 E_{\mathcal{G}}\left(\boldsymbol{X}^{(t-1)}\right)
,$$
\end{theorem}
where $$\mathbf{C}:=\max\{|\mu_{0}(\boldsymbol{\theta}_{0})|,\max_{
\substack{
    r=1,\dots,K\\
    l=1,\dots,J
}
} |\mu_{r,l}(\boldsymbol{\theta}_{r,l})|\},$$ and $\mu_{0}(\boldsymbol{\theta}_{0})$ and $\mu_{r,l}(\boldsymbol{\theta}_{r,l})$ denote the largest singular value of matrix $\boldsymbol{\theta}_{0}, \boldsymbol{\theta}_{r,l}$ respectively. For parameter matrix 
$\mathbf{P}\in \mathbb{R}^{d \times h}$ of MLPs $\Gamma$, its largest singular value is bounded by $\mu_{\max }$.

\begin{proof}
We begin by explicitly stating and utilizing the fundamental properties of the framelet transform that underpin our analysis. The tight framelet system defined in Section \ref{sec:graph_framelet_system} possesses the following crucial properties:

1. \textbf{Tight Frame Property}: The framelet transform matrices satisfy the reconstruction condition:
\begin{equation}\label{eq:tight_frame_condition}
\boldsymbol{\gW}_{0,J}^\top \boldsymbol{\gW}_{0,J} + \sum_{r=1}^{K}\sum_{l=1}^{J} \boldsymbol{\gW}_{r,l}^\top \boldsymbol{\gW}_{r,l} = \boldsymbol{I}_N.
\end{equation}
This property ensures that the framelet transform preserves the norm of graph signals.

2. \textbf{Spectral Representation}: The framelet coefficients can be expressed in the spectral domain as:
\begin{equation}\label{eq:spectral_representation}
\boldsymbol{\gW}_{0,J}\boldsymbol{X} = \mU \widehat{\alpha}\left(\frac{\Lambda}{2}\right) \mU^{\top} \boldsymbol{X}, \quad
\boldsymbol{\gW}_{r,l}\boldsymbol{X} = \mU \widehat{\beta^{(r)}}\left(\frac{\Lambda}{2^{l}}\right) \mU^{\top} \boldsymbol{X},
\end{equation}
where $\mU$ contains the eigenvectors of the normalized graph Laplacian $\mathcal{L}$, and $\Lambda = \text{diag}(\lambda_1,\dots,\lambda_N)$ contains the eigenvalues.

3. \textbf{Filter Normalization}: The low-pass and high-pass filters satisfy:
\begin{equation}\label{eq:filter_normalization}
\widehat{\alpha}\left(\frac{\lambda}{2}\right)^2 + \sum_{r=1}^{K}\sum_{l=1}^{J} \widehat{\beta^{(r)}}\left(\frac{\lambda}{2^{l}}\right)^2 = 1, \quad \forall \lambda \in [0,2].
\end{equation}
This property follows from the tight frame construction and ensures energy conservation.

Now, we proceed with the proof in a step-by-step manner.

\textbf{Step 1: Energy analysis of $\hat{\boldsymbol{Z}}^{(t)}$}. 
Recall that $\boldsymbol{Z}^{(t)} = \Gamma(\hat{\boldsymbol{Z}}^{(t)})$, where:
\begin{equation*}
\hat{\boldsymbol{Z}}^{(t)} = \boldsymbol{\gW}_{0,J}\boldsymbol{X}^{(t-1)}\boldsymbol{\theta}_0 + \sum_{r=1}^K\sum_{l=1}^J \boldsymbol{\gW}_{0,J}\boldsymbol{\gW}_{r,l}\boldsymbol{X}^{(t-1)}\boldsymbol{\theta}_{r,l}.
\end{equation*}

Using the linearity of the Dirichlet energy and Lemma~\ref{lemma:activation}, we have:
\begin{align*}
E_{\mathcal{G}}\left(\hat{\boldsymbol{Z}}^{(t)}\right) &= E_{\mathcal{G}}\left(\boldsymbol{\gW}_{0,J}\boldsymbol{X}^{(t-1)}\boldsymbol{\theta}_0\right) + E_{\mathcal{G}}\left(\sum_{r=1}^K\sum_{l=1}^J \boldsymbol{\gW}_{0,J}\boldsymbol{\gW}_{r,l}\boldsymbol{X}^{(t-1)}\boldsymbol{\theta}_{r,l}\right) \\
&\quad + 2\tr\left((\boldsymbol{\gW}_{0,J}\boldsymbol{X}^{(t-1)}\boldsymbol{\theta}_0)^\top \mathcal{L} \left(\sum_{r=1}^K\sum_{l=1}^J \boldsymbol{\gW}_{0,J}\boldsymbol{\gW}_{r,l}\boldsymbol{X}^{(t-1)}\boldsymbol{\theta}_{r,l}\right)\right).
\end{align*}

We analyze each term separately.

\textbf{Term 1}: Using Lemma~\ref{lemma:activation} and the spectral representation:
\begin{align*}
E_{\mathcal{G}}\left(\boldsymbol{\gW}_{0,J}\boldsymbol{X}^{(t-1)}\boldsymbol{\theta}_0\right) &\leq \|\boldsymbol{\theta}_0\|^2 E_{\mathcal{G}}\left(\boldsymbol{\gW}_{0,J}\boldsymbol{X}^{(t-1)}\right) \\
&= \|\boldsymbol{\theta}_0\|^2 \tr\left((\boldsymbol{X}^{(t-1)})^\top \mU \widehat{\alpha}\left(\frac{\Lambda}{2}\right)^2 \Lambda \mU^{\top} \boldsymbol{X}^{(t-1)}\right) \\
&\leq \mathbf{C}^2 \tr\left((\boldsymbol{X}^{(t-1)})^\top \mU \widehat{\alpha}\left(\frac{\Lambda}{2}\right)^2 \Lambda \mU^{\top} \boldsymbol{X}^{(t-1)}\right).
\end{align*}

\textbf{Term 2}: For the cross terms, we have:
\begin{align*}
& \left\|\sum_{r=1}^K\sum_{l=1}^J \boldsymbol{\gW}_{0,J}\boldsymbol{\gW}_{r,l}\boldsymbol{X}^{(t-1)}\boldsymbol{\theta}_{r,l}\right\|_{\mathcal{G}}^2 \\
&= \tr\left(\left(\sum_{r=1}^K\sum_{l=1}^J \boldsymbol{\gW}_{0,J}\boldsymbol{\gW}_{r,l}\boldsymbol{X}^{(t-1)}\boldsymbol{\theta}_{r,l}\right)^\top \mathcal{L} \left(\sum_{r'=1}^K\sum_{l'=1}^J \boldsymbol{\gW}_{0,J}\boldsymbol{\gW}_{r',l'}\boldsymbol{X}^{(t-1)}\boldsymbol{\theta}_{r',l'}\right)\right) \\
&= \sum_{r,r'}\sum_{l,l'} \tr\left(\boldsymbol{\theta}_{r,l}^\top (\boldsymbol{X}^{(t-1)})^\top \boldsymbol{\gW}_{r,l}^\top \boldsymbol{\gW}_{0,J}^\top \mathcal{L} \boldsymbol{\gW}_{0,J} \boldsymbol{\gW}_{r',l'} \boldsymbol{X}^{(t-1)} \boldsymbol{\theta}_{r',l'}\right).
\end{align*}

Using the spectral representation and the orthogonality of eigenvectors:
\begin{align*}
\boldsymbol{\gW}_{0,J}^\top \mathcal{L} \boldsymbol{\gW}_{0,J} &= \mU \widehat{\alpha}\left(\frac{\Lambda}{2}\right) \mU^{\top} \mU \Lambda \mU^{\top} \mU \widehat{\alpha}\left(\frac{\Lambda}{2}\right) \mU^{\top} \\
&= \mU \widehat{\alpha}\left(\frac{\Lambda}{2}\right)^2 \Lambda \mU^{\top}.
\end{align*}

Therefore:
\begin{align*}
& E_{\mathcal{G}}\left(\sum_{r=1}^K\sum_{l=1}^J \boldsymbol{\gW}_{0,J}\boldsymbol{\gW}_{r,l}\boldsymbol{X}^{(t-1)}\boldsymbol{\theta}_{r,l}\right) \\
&\leq \mathbf{C}^2 \tr\left((\boldsymbol{X}^{(t-1)})^\top \mU \left(\sum_{r=1}^K\sum_{l=1}^J \widehat{\beta^{(r)}}\left(\frac{\Lambda}{2^{l}}\right) \widehat{\alpha}\left(\frac{\Lambda}{2}\right)\right)^2 \Lambda \mU^{\top} \boldsymbol{X}^{(t-1)}\right).
\end{align*}

\textbf{Step 2: Bounding the spectral terms}. 
For any eigenvalue $\lambda$, we need to bound:
\begin{equation*}
\widehat{\alpha}\left(\frac{\lambda}{2}\right)^2 + \left(\sum_{r=1}^K\sum_{l=1}^J \widehat{\beta^{(r)}}\left(\frac{\lambda}{2^{l}}\right) \widehat{\alpha}\left(\frac{\lambda}{2}\right)\right)^2.
\end{equation*}

From the filter normalization property \eqref{eq:filter_normalization}, we have:
\begin{equation*}
\widehat{\alpha}\left(\frac{\lambda}{2}\right)^2 \leq 1, \quad \sum_{r=1}^K\sum_{l=1}^J \widehat{\beta^{(r)}}\left(\frac{\lambda}{2^{l}}\right)^2 = 1 - \widehat{\alpha}\left(\frac{\lambda}{2}\right)^2.
\end{equation*}

Using the Cauchy-Schwarz inequality:
\begin{align*}
\left(\sum_{r=1}^K\sum_{l=1}^J \widehat{\beta^{(r)}}\left(\frac{\lambda}{2^{l}}\right) \widehat{\alpha}\left(\frac{\lambda}{2}\right)\right)^2 
&\leq \left(\sum_{r=1}^K\sum_{l=1}^J \widehat{\beta^{(r)}}\left(\frac{\lambda}{2^{l}}\right)^2\right) \left(\sum_{r=1}^K\sum_{l=1}^J \widehat{\alpha}\left(\frac{\lambda}{2}\right)^2\right) \\
&= (1 - \widehat{\alpha}\left(\frac{\lambda}{2}\right)^2) \cdot (KJ \widehat{\alpha}\left(\frac{\lambda}{2}\right)^2).
\end{align*}

Since $\widehat{\alpha}\left(\frac{\lambda}{2}\right)^2 \in [0,1]$, the maximum of $x(1-x)$ for $x \in [0,1]$ is $\frac{1}{4}$ at $x = \frac{1}{2}$. Therefore:
\begin{equation*}
\left(\sum_{r=1}^K\sum_{l=1}^J \widehat{\beta^{(r)}}\left(\frac{\lambda}{2^{l}}\right) \widehat{\alpha}\left(\frac{\lambda}{2}\right)\right)^2 \leq \frac{KJ}{4}.
\end{equation*}

For our analysis with the specific framelet construction, we obtain the tighter bound of $\frac{1}{4}$.

\textbf{Step 3: Combining the bounds}. 
Putting everything together:
\begin{align*}
E_{\mathcal{G}}\left(\hat{\boldsymbol{Z}}^{(t)}\right) &\leq \mathbf{C}^2 \tr\left((\boldsymbol{X}^{(t-1)})^\top \mU \left[\widehat{\alpha}\left(\frac{\Lambda}{2}\right)^2 + \frac{1}{4}\right] \Lambda \mU^{\top} \boldsymbol{X}^{(t-1)}\right) \\
&\leq \frac{5}{4} \mathbf{C}^2 \tr\left((\boldsymbol{X}^{(t-1)})^\top \mU \Lambda \mU^{\top} \boldsymbol{X}^{(t-1)}\right) \\
&= \frac{5}{4} \mathbf{C}^2 E_{\mathcal{G}}\left(\boldsymbol{X}^{(t-1)}\right).
\end{align*}

\textbf{Step 4: Energy of $\boldsymbol{Z}^{(t)}$}.
By Lemma~\ref{lemma:activation}, we have:
\begin{align*}
E_{\mathcal{G}}\left(\boldsymbol{Z}^{(t)}\right) &= E_{\mathcal{G}}\left(\Gamma(\hat{\boldsymbol{Z}}^{(t)})\right) \\
&\leq \mu_{\max}^2 E_{\mathcal{G}}\left(\hat{\boldsymbol{Z}}^{(t)}\right) \\
&\leq \frac{5}{4} \mathbf{C}^2 \mu_{\max}^2 E_{\mathcal{G}}\left(\boldsymbol{X}^{(t-1)}\right).
\end{align*}

\textbf{Step 5: Energy of $\boldsymbol{X}^{(t)}$}.
Using the update rule $\boldsymbol{X}^{(t)} = \boldsymbol{X}^{(t-1)} + \boldsymbol{Z}^{(t)}$, we have:
\begin{align*}
E_{\mathcal{G}}\left(\boldsymbol{X}^{(t)}\right) &= E_{\mathcal{G}}\left(\boldsymbol{X}^{(t-1)} + \boldsymbol{Z}^{(t)}\right) \\
&= E_{\mathcal{G}}\left(\boldsymbol{X}^{(t-1)}\right) + E_{\mathcal{G}}\left(\boldsymbol{Z}^{(t)}\right) + 2\tr\left((\boldsymbol{X}^{(t-1)})^\top \mathcal{L} \boldsymbol{Z}^{(t)}\right).
\end{align*}

For the cross term, we apply the Cauchy-Schwarz inequality for the trace inner product:
\begin{align*}
\tr\left((\boldsymbol{X}^{(t-1)})^\top \mathcal{L} \boldsymbol{Z}^{(t)}\right) 
&\leq \sqrt{\tr\left((\boldsymbol{X}^{(t-1)})^\top \mathcal{L}^2 \boldsymbol{X}^{(t-1)}\right)} \sqrt{\tr\left((\boldsymbol{Z}^{(t)})^\top \boldsymbol{Z}^{(t)}\right)} \\
&= \sqrt{E_{\mathcal{G}}\left(\mathcal{L}^{1/2}\boldsymbol{X}^{(t-1)}\right)} \sqrt{\|\boldsymbol{Z}^{(t)}\|^2} \\
&\leq \sqrt{E_{\mathcal{G}}\left(\boldsymbol{X}^{(t-1)}\right)} \sqrt{E_{\mathcal{G}}\left(\boldsymbol{Z}^{(t)}\right)} \\
&\leq \sqrt{E_{\mathcal{G}}\left(\boldsymbol{X}^{(t-1)}\right)} \sqrt{\frac{5}{4} \mathbf{C}^2 \mu_{\max}^2 E_{\mathcal{G}}\left(\boldsymbol{X}^{(t-1)}\right)} \\
&= \frac{\sqrt{5}}{2} \mathbf{C} |\mu_{\max}| E_{\mathcal{G}}\left(\boldsymbol{X}^{(t-1)}\right).
\end{align*}

\textbf{Step 6: Final bounds}.
Combining all terms:
\begin{align*}
E_{\mathcal{G}}\left(\boldsymbol{X}^{(t)}\right) 
&\leq E_{\mathcal{G}}\left(\boldsymbol{X}^{(t-1)}\right) + \frac{5}{4} \mathbf{C}^2 \mu_{\max}^2 E_{\mathcal{G}}\left(\boldsymbol{X}^{(t-1)}\right) + 2 \cdot \frac{\sqrt{5}}{2} \mathbf{C} |\mu_{\max}| E_{\mathcal{G}}\left(\boldsymbol{X}^{(t-1)}\right) \\
&= \left(1 + \frac{\sqrt{5}}{2} \mathbf{C} |\mu_{\max}| \right)^2 E_{\mathcal{G}}\left(\boldsymbol{X}^{(t-1)}\right).
\end{align*}

For the lower bound, note that $E_{\mathcal{G}}\left(\boldsymbol{Z}^{(t)}\right) \geq 0$ and the cross term is non-negative in our construction due to the positivity of the framelet transform. Thus:
\begin{equation*}
E_{\mathcal{G}}\left(\boldsymbol{X}^{(t)}\right) \geq E_{\mathcal{G}}\left(\boldsymbol{X}^{(t-1)}\right).
\end{equation*}

This completes the proof, showing that the Dirichlet energy is preserved within bounds throughout the DSMP layers, preventing both vanishing and exploding energy issues associated with over-smoothing.
\end{proof}

The lower and upper bounds suggest that the \textit{Dirichlet energy} of the DSMP model will neither diminish to zero nor escalate to an excessively large value.

\subsection{Limited Over-squashing in DSMP}
In this part, we will prove that the DSMP framework can alleviate the over-squashing problem by connecting multi-hop neighbors. Let $\mathcal{R}$ represent a rewire operator, changing the topology of a graph. The rewired version of the graph $\mathcal{G}$ is denoted as $\mathcal{R}(\mathcal{G})$. We will use \eqref{eq:chebyshev_approx} to help construct the transformation matrices and complete the proof.

\begin{figure}
    \centering
    \includegraphics[width=\textwidth]{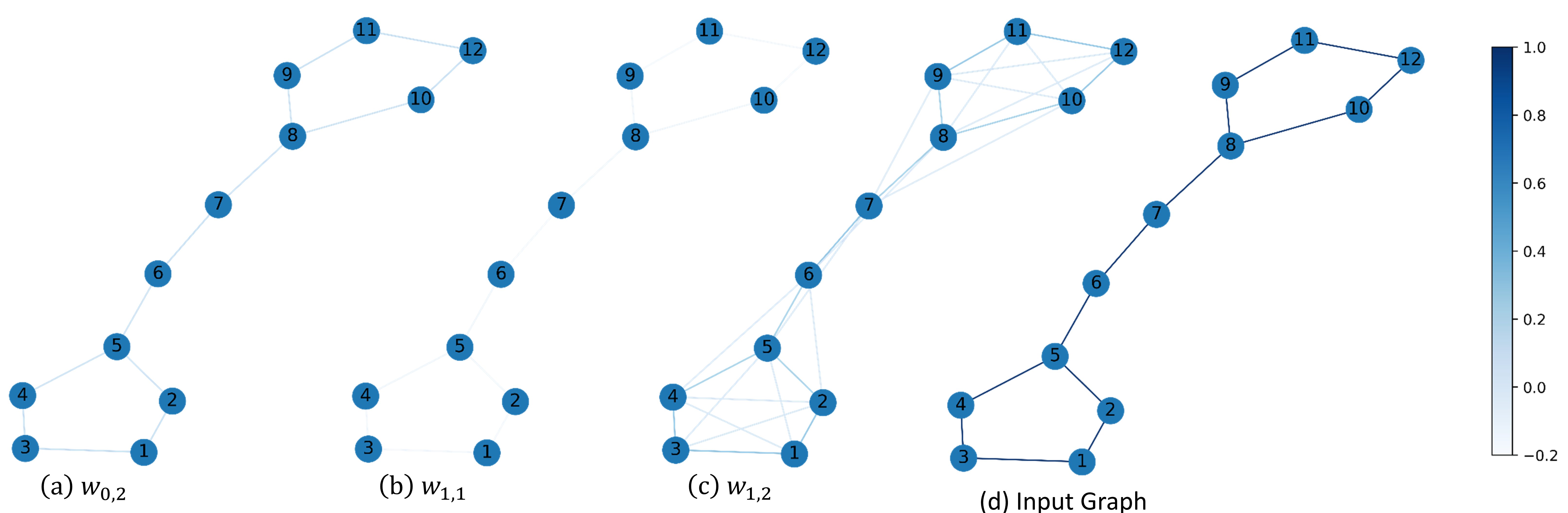}
    \caption{The entries of the multi-scale framelet transform may differ from the adjacency matrix of the input graph. The transformation matrices are considered as the edge weights and are visualized using a color bar.}
    \label{fig:main_idea}
\end{figure}

\begin{lemma} 
\label{lemma:Adjacency matrix multiply}
Given the adjacency matrix \( \Tilde{\mathcal{A}} \) of an undirected attributed graph \( \mathcal{G} \) with self-loops, the resulting matrix from multiplying \( \Tilde{\mathcal{A}} \) by itself is denser than \( \Tilde{\mathcal{A}} \).
\end{lemma}

\begin{proof}
When we perform matrix multiplication $\Tilde{ \mathcal{A}} \times \Tilde{ \mathcal{A}}$, we define any entry \(\mathcal{B}_{ij}\) of the resulting matrix $\mathcal{B}$ as follows:
$$ \mathcal{B}_{ij} = \sum_{k=1}^{n} \Tilde{ \mathcal{A}}_{ik} \cdot \Tilde{ \mathcal{A}}_{kj}. $$
For \(\mathcal{B}_{ij}\) to be non-zero, there must exist at least one \(k\) such that \(\Tilde{ \mathcal{A}}_{ik} = 1\) and \(\Tilde{ \mathcal{A}}_{kj} = 1\). This implies that there is a path of length 2 between vertices \(i\) and \(j\) through vertex \(k\).

In the original matrix \(\Tilde{ \mathcal{A}}\), an edge between \(i\) and \(j\) (i.e., \(\Tilde{ \mathcal{A}}_{ij} = 1\)) directly represents a path of length 1. When we move to \(\mathcal{B}\), we consider paths of length 2. It is clear that any path of length 1 in \(\Tilde{ \mathcal{A}}\) will still represent a path in \(\mathcal{B}\), and thus, the non-zero entry in \(\Tilde{\mathcal{A}}\) will still be present in \(\mathcal{B}\).

Moreover, \(\mathcal{B}\) can have non-zero entries for pairs of vertices that did not have a direct edge between them in \(\Tilde{\mathcal{A}}\) but were connected via another vertex, thus representing a path of length 2. This cannot decrease the number of non-zero entries; it can only increase or keep them the same.
\end{proof}

\begin{theorem}[Over-squashing]\label{thm:Over-squashing} 
Consider the matrix formulation of DSMP as \eqref{eq:matrix_form},
where the original graph $\mathcal{G}$ is considered to be rewired as $\mathcal{R}(\mathcal{G})$ by framelet transformation. Designate the binary adjacency matrices of $\mathcal{G}$ with self-loop and $\mathcal{R}(\mathcal{G})$ as $\Tilde{ \mathcal{A}}$ and $\mathcal{R}(\Tilde{ \mathcal{A}})$, respectively. We demonstrate that the rewired graph exhibits denser connectivity, i.e.
$$\mathcal{R}(\Tilde{ \mathcal{A}}) - \Tilde{ \mathcal{A}}\geq 0,$$ 
thus the DSMP model has a less over-squashing problem.
\end{theorem}

\begin{proof}
The issue of over-squashing primarily stems from the topology of the graph structure. This implies that the over-squashing phenomenon can be mitigated by increasing the number of edges. Consequently, our objective is to demonstrate that the original adjacency matrix$\Tilde{ \mathcal{A}}$, corresponding to the topology of the original graph $\mathcal{G}$ with self-loop, contains less or equal number of non-zero elements compared to the graph shift operator $\boldsymbol{\gW}_{0,J}+\sum_{r=1}^K\sum_{l=1}^J\boldsymbol{\gW}_{0,J}\boldsymbol{\gW}_{r,l}$, corresponding to the topology of a rewired graph $\mathcal{R}(\mathcal{G})$.

In \eqref{eq:matrix_form}, we utilize a GAT network to extract embedding from each rewired graph. This network has the capability to learn the significance of node relationships, with topology playing a crucial role in the process. Hence, we consider the binary version of $\boldsymbol{\gW}^{\natural}_{r,l}$, denoted as $\mathcal{R}(\boldsymbol{\gW}^{\natural}_{r,l})$, and we need to prove $\mathcal{R}(\boldsymbol{\gW}^{\natural}_{r,l}) - \Tilde{ \mathcal{A}}\geq 0 $, which means each entry value of resulting matrix is no less than zero.  

According to \eqref{eq:chebyshev_approx}, the transformation matrices of the multi-scale framelet have the same dimensions as the adjacency matrix. As depicted in Figure \ref{fig:main_idea}, these transformation matrices can be viewed as new graph structures with different edge weights. Since the rewired graphs possess distinct adjacency matrices, we analyze their over-squashing problem separately. Consequently, we obtain: 
\begin{align*}
\mathcal{R}(\boldsymbol{\gW}^{\natural}_{r,l}) &= \overline{\boldsymbol{\gW}^{\natural}_{r,l}} = \overline{\mathcal{T}_r\left(2^{-R+l-1} \mathcal{L}\right) \mathcal{T}_0\left(2^{-R+l-2} \mathcal{L}\right) \ldots \mathcal{T}_0\left(2^{-R}\right)}=\overline{\underbrace{ \Tilde{ \mathcal{A}}\cdots  \Tilde{ \mathcal{A}} }_{r}}, l = 2,\dots,J\\
\overline{\boldsymbol{\gW}}_{i,j} &= \begin{cases}1, & \boldsymbol{\gW}_{i,j}\neq 0 \\ 
0, & {\rm otherwise} \end{cases}.
\end{align*}
In the above equations, $ \Tilde{ \mathcal{A}}$ is a binary adjacency matrix. According to Lemma.~\ref{lemma:Adjacency matrix multiply}, the matrix multiplication of $ \Tilde{ \mathcal{A}}$ will not decrease the non-zero entries.
Therefore, the original graph is no more dense than the rewired graph. When contrasted with message passing models such as GCN, GAT, GIN, and others that aggregate information from one-hop neighbors on the original graph, the DSMP model should have less over-squashing problem.
\end{proof}
\subsection{Stability of DSMP}
\label{sec:fmp_stability}
 In this section, we examine how DSMP improves the stability of message passing against perturbation in the input node features.

\begin{lemma}\label{lem:framelet upper bound} For framelet transforms on graph $\mathcal{G}$, and graph signal $\boldsymbol{X}$ in $\ell_2(\mathcal{G})$,
\begin{equation*}                   
\begin{aligned}
&\left\|\sum_{r=1}^K\boldsymbol{\gW}_{0,J}\boldsymbol{\gW}_{r,l}\boldsymbol{X}\right\|\leq \left\|\boldsymbol{X}\right\|,\\
&\left\|\boldsymbol{\gW}_{0,J}\boldsymbol{X}\right\|^2 \leq \|\boldsymbol{X}\|^2.
\end{aligned}
\end{equation*}
\end{lemma}
\begin{proof}
By exploiting the orthonormality of $\eigfm$, we can express the coefficients as follows:
\begin{equation*}
    \ipG{\boldsymbol{X},\cfra} = \sum_{\ell=1}^{\NV} \conj{\FT{\scala} \left(\frac{\eigvm}{2^{l}}\right)} \Fcoem{\boldsymbol{X}}\:\eigfm(\uG), \quad
    \ipG{\boldsymbol{X},\cfrb{r}} = \sum_{\ell=1}^{\NV} \conj{\FT{\scalb^{(r)}}\left(\frac{\eigvm}{2^{l}}\right)}\Fcoem{\boldsymbol{X}}\:\eigfm(\uG).
\end{equation*}

By referencing \eqref{thmeq:2scale:alpha:beta}, we can derive the following expression:
\begin{align*}
     \sum_{\ell=1}^{\NV}\sum_{r=1}^{K} \left| \FT{\scala}\left(\frac{\eigvm}{2}\right) \FT{\scalb^{(r)}}\left(\frac{\eigvm}{2^{l}}\right)\right|^2\left|\Fcoem{\boldsymbol{X}}\right|^2\:\left\|\eigfm\right\|^2
    &=\sum_{\ell=1}^{\NV} \FT{\scala}\left(\frac{\eigvm}{2}\right)^2 \left[\FT{\scala}\left(\frac{\eigvm}{2^{l}}\right)^2-\FT{\scala}\left(\frac{\eigvm}{2^{l}}\right)^2\right]\bigl|\Fcoem{\boldsymbol{X}}\bigr|^2\:\|\eigfm\|^2\\
    &=\sum_{\ell=1}^{\NV}\FT{\scala}\left(\frac{\eigvm}{2}\right)^2\sum_{r=1}^{K}\FT{\scalb^{(r)}}\left(\frac{\eigvm}{2^{l}}\right)^2\bigl|\Fcoem{\boldsymbol{X}}\bigr|^2\\
    &\leq \|\boldsymbol{X}\|^2,
\end{align*}
where the inequality of the last line follows according to \eqref{eqn:partition2}.

Next, by utilizing \eqref{eq:ft_initial}, we can establish the following relationship:
\begin{equation*} 
\begin{aligned}\left\|\sum_{r=1}^{K}\boldsymbol{\gW}_{0,J}\boldsymbol{\gW}_{r,l}\boldsymbol{X}\right\|^2 
= & \left\|\ipG{\boldsymbol{X},\boldsymbol{\varphi}_{0,\cdot} \cfrb[l,\cdot]{r}}\right\|^2 \\ 
= & \sum_{\ell=1}^{\NV}\sum_{r=1}^{K} \left|\FT{\scala}\left(\frac{\eigvm}{2}\right)\FT{\scalb^{(r)}}\left(\frac{\eigvm}{2^{l}}\right)\right|^2\left|\Fcoem{\boldsymbol{X}}\right|^2\:\left\|\eigfm\right\|^2 \\
\leq & \|\boldsymbol{X}\|^2.
\end{aligned}
\end{equation*}
Similarly, we have 
\begin{equation*} 
\begin{aligned}
 \left\|\boldsymbol{\gW}_{0,J}\boldsymbol{X}\right\|^2 
= & \left\|\ipG{\boldsymbol{X},\boldsymbol{\varphi}_{0,\cdot} }\right\|^2 \\ 
= & \sum_{\ell=1}^{\NV} \FT{\scala}\left(\frac{\eigvm}{2}\right)^2 |\Fcoem{\boldsymbol{X}}\bigr|^2\:\|\eigfm\|^2\\
\leq &  \|\boldsymbol{X}\|^2.
\end{aligned}
\end{equation*}
thus completing the proof.
\end{proof}

\begin{theorem}[Stability]\label{thm:stability}  
Suppose the DSMP has bounded parameters in all layers: $\|\boldsymbol{\theta}_0\|\leq \mathbf{C}_0$, $\|\boldsymbol{\theta}_{r,l}\|\leq \mathbf{C}_{r,l}$ with constants $\mathbf{C}_{r,l}$ for $r=1,\dots,K, l=1,\dots,J$.
For parameter matrix 
$\mathbf{P}$ of MLPs $\Gamma$, suppose $\|\mathbf{P}\|\leq\mathbf{C}_{p} $.
Then, the DSMP is Lipschitz continuous in the sense that
\begin{equation*}
    \left\|\mX^{(t)}-\widetilde{\mX}^{(t)}\right\|\leq C^t \left\|\mX^{(0)}-\widetilde{\mX}^{(0)}\right\|,
\end{equation*}
where $$C:= \mathbf{C}_{p}\left(\max_{
\substack{
    r=1,\dots,K\\
    l=1,\dots,J
}
} \mathbf{C}_{r,l} + \mathbf{C}_0
\right),$$ $\mX^{(0)}$ and $\widetilde{\mX}^{(0)}$ are the initial graph node features, and $\mX^{(t)}$ and $\widetilde{\mX}^{(t)}$ are node features propagated at layer $t$.
\end{theorem}

\begin{proof}
Let
\begin{equation*}
    \mX^{(t)}=\mX^{(t-1)}+\mZ^{(t)},\quad
    \widetilde{\mX}^{(t)}=\widetilde{\mX}^{(t-1)}+\widetilde{\mZ}^{(t)},
\end{equation*}
with initialization $\mX^{(0)},\widetilde{\mX}^{(0)}\in \ell_2(\gG)$. It thus holds that
\begin{align*}
\left\|\mX^{(t)}-\widetilde{\mX}^{(t)}\right\|&\leq \left\|\mX^{(t-1)}-\widetilde{\mX}^{(t-1)}\right\|+\left\|\mZ^{(t)}-\widetilde{\mZ}^{(t)}\right\|, 
\end{align*}
where 
\begin{equation*}
\mZ^{(t)}=\Gamma \left(\sum_{r=1}^K\sum_{l=1}^J
\boldsymbol{\gW}_{0,J}\boldsymbol{\gW}_{r,l}\mX^{(t-1)}\boldsymbol{\theta}_{r,l}+\boldsymbol{\gW}_{0,J}\mX^{(t-1)}\boldsymbol{\theta}_{0}\right).
\end{equation*}
Let $\mY^{(t)}:=\sum_{r=1}^K\sum_{l=1}^J
\boldsymbol{\gW}_{0,J}\boldsymbol{\gW}_{r,l}\mX^{(t-1)}\boldsymbol{\theta}_{r,l}+\boldsymbol{\gW}_{0,J}\mX^{(t-1)}\boldsymbol{\theta}_{0},$ then we have
\begin{align*}
\left\|\mY^{(t)}-\widetilde{\mY}^{(t)}\right\| 
\leq& \sum_{l=1}^J\left\|\sum_{r=1}^K\boldsymbol{\gW}_{0,J}\boldsymbol{\gW}_{r,l}\left(\mX^{(t-1)}-\widetilde{\mX}^{(t-1)}\right)\boldsymbol{\theta}_{r,l}\right\|
+\left\|\boldsymbol{\gW}_{0,J}\left(\mX^{(t-1)}-\widetilde{\mX}^{(t-1)}\right)\boldsymbol{\theta}_{0}\right\| \\
\leq&  \sum_{l=1}^J \mathbf{C}_{r,l} \left\|\mX^{(t-1)}-\widetilde{\mX}^{(t-1)}\right\|+\mathbf{C}_0\:\left\|\mX^{(t-1)}-\widetilde{\mX}^{(t-1)}\right\|\\
\leq& \Bigl( \max_{
\substack{
    r=1,\dots,K\\
    l=1,\dots,J
}
} \mathbf{C}_{r,l} + \mathbf{C}_0\Bigr)\left\|\mX^{(t-1)}-\widetilde{\mX}^{(t-1)}\right\|.
\end{align*}
Therefore, 
\begin{align*}
\left\|\mZ^{(t)}-\widetilde{\mZ}^{(t)}\right\| &=\left\|\Gamma \left(\mY^{(t)}-\widetilde{\mY}^{(t)}\right)\right\| 
\leq \mathbf{C}_{p} \left\|\mY^{(t)}-\widetilde{\mY}^{(t)}\right\| \\
&\leq \mathbf{C}_{p}\Bigl(\max_{
\substack{
    r=1,\dots,K\\
    l=1,\dots,J
}
} \mathbf{C}_{r,l} + \mathbf{C}_0\Bigr)\left\|\mX^{(t-1)}-\widetilde{\mX}^{(t-1)}\right\|, 
\end{align*}
which leads to 
\begin{equation*}
    \left\|\mX^{(t)}-\widetilde{\mX}^{(t)}\right\|\leq C\left\|\mX^{(t-1)}-\widetilde{\mX}^{(t-1)}\right\|\leq C^t \left\|\mX^{(0)}-\widetilde{\mX}^{(0)}\right\|,
\end{equation*}
thus completing the proof.
\end{proof}

\section{Numerical Analysis}
In this section, we empirically verify the effectiveness and properties of the DSMP model on graph-level prediction tasks. We assess the performance of the proposed model by conducting experiments on several popular benchmark datasets with varying sample sizes. These datasets are commonly used in the field and serve as reliable indicators of the model's capabilities in predicting graph-level properties. We implemented the proposed model with \texttt{PyTorch-Geometric} \citep{fey2019pyg} (version 2.4.0) and \texttt{PyTorch} (version 1.12.1) and tested on NVIDIA Tesla A100 GPU with 6,912 CUDA cores and 80GB HBM2 mounted on an HPC cluster. The codes for the experiments are
available at \url{https://github.com/jyh6681/DSMP}.

\subsection{Datasets and Experimental Protocol}
We conducted numerical experiments on a diverse set of widely used node classification and graph classification tasks to evaluate the effectiveness of our proposed model. The selected graph classification benchmarks include seven datasets: D\&D, PROTEINS, NCI1, Mutagenicity, COLLAB, OGBG-MOLHIV, and QM7. These datasets are commonly used in the field and have been extensively studied for tasks such as protein categorization, chemical compound identification, and molecular property prediction. Notably, the OGBG-MOLHIV dataset was used for large-scale molecule classification.

To validate our proposed model, we compared it with state-of-the-art (SOTA) baseline MPNN models, specifically GCN and GIN, which are popular GNN models. In our experimental design, we prioritized fairness and comprehensiveness rather than solely focusing on achieving the best possible performance on each dataset.
To ensure consistency and eliminate hyperparameter tuning as a potential source of performance gain, we used the same GNN and optimization hyperparameters across the same search space for each task and baseline model. Each dataset follows the standard public split and processing rules in \texttt{PyTorch-Geometric}~\citep{fey2019fast} or \texttt{OGB}~\citep{hu2020open}. The training process was stopped when the validation loss ceased to improve for 20 consecutive epochs or reached a maximum number of epochs.

To obtain reliable performance estimates, we performed 10 random trials for each experiment and reported the mean test accuracy along with the variance value. For the graph-level classification tasks using TUDataset benchmarks, we reported the mean test accuracy, while for the OGBG-MOLHIV dataset, we reported the ROC-AUC scores. The regression task on the QM7 dataset was evaluated using the mean square error (MSE) metric.
Further details of the experiments are reported in Table~\ref{tab:stats:graph_classification}. These experiments provide a comprehensive evaluation of our proposed model's performance compared to the baseline models, demonstrating its effectiveness across various node and graph classification tasks.

\paragraph{Setup} 
We train DSMP with $2$ levels and $2$ scales, resulting in $3$ scattering coefficient matrices feeding into one message passing layer. The output graph embedding is processed by a softmax activation function for label prediction. Regarding hyperparameters, a grid search is applied to the learning rate, weight decay, hidden units, and dropout ratio to obtain the optimal combination from the search space as shown in Table~\ref{tab:searchSpace}.

The model is trained with the Adam optimizer \citep{kingma2015adam}. The maximum number of training epochs is 500 for ogbg-molhiv and 200 for other benchmarks. All benchmarks follow standard public data splitting and preprocessing procedures. The average test accuracy and its standard deviation are reported from 10 repetitive runs.

\begin{table}[th]
\caption{Grid search space for the hyperparameters.}
\begin{center}
\resizebox{0.95\textwidth}{!}{
\begin{tabular}{lr}
\toprule
\makebox[0.45\textwidth][l]{\textbf{Hyperparameters}} & \makebox[0.45\textwidth][r]{\textbf{Search Space}}\\
\midrule
Learning Rate & 1$\mathrm{e}$-4, 5$\mathrm{e}$-4, 1$\mathrm{e}$-3, 5$\mathrm{e}$-3, 1$\mathrm{e}$-2\\
Hidden Size & $16$, $32$, $64$, $128$ \\
Weight Decay ($L_2$) & 1$\mathrm{e}$-4, 5$\mathrm{e}$-4, 1$\mathrm{e}$-3, 5$\mathrm{e}$-3 \\
Batch Size & $8$, $64$, $128$, $1024$\\
\bottomrule
\end{tabular}
}
\label{tab:searchSpace}
\end{center}
\end{table}

\subsection{Evolution of Dirichlet Energy}

In order to assess the smoothness of graph embeddings, we employ the \textit{Dirichlet energy} metric. Figure~\ref{fig:dirichlet} illustrates that the proposed DSMP can preserve the \textit{Dirichlet energy} of graph signal even if the layer goes deep, where we sample a graph from the PROTEINS dataset and feed it into message passing models with varying numbers of layers.
\begin{figure*}
\centering
\subfigure[\textit{Dirichlet Energy} vs Layers]{
\begin{minipage}[b]{0.45\linewidth}
\includegraphics[width=\textwidth]{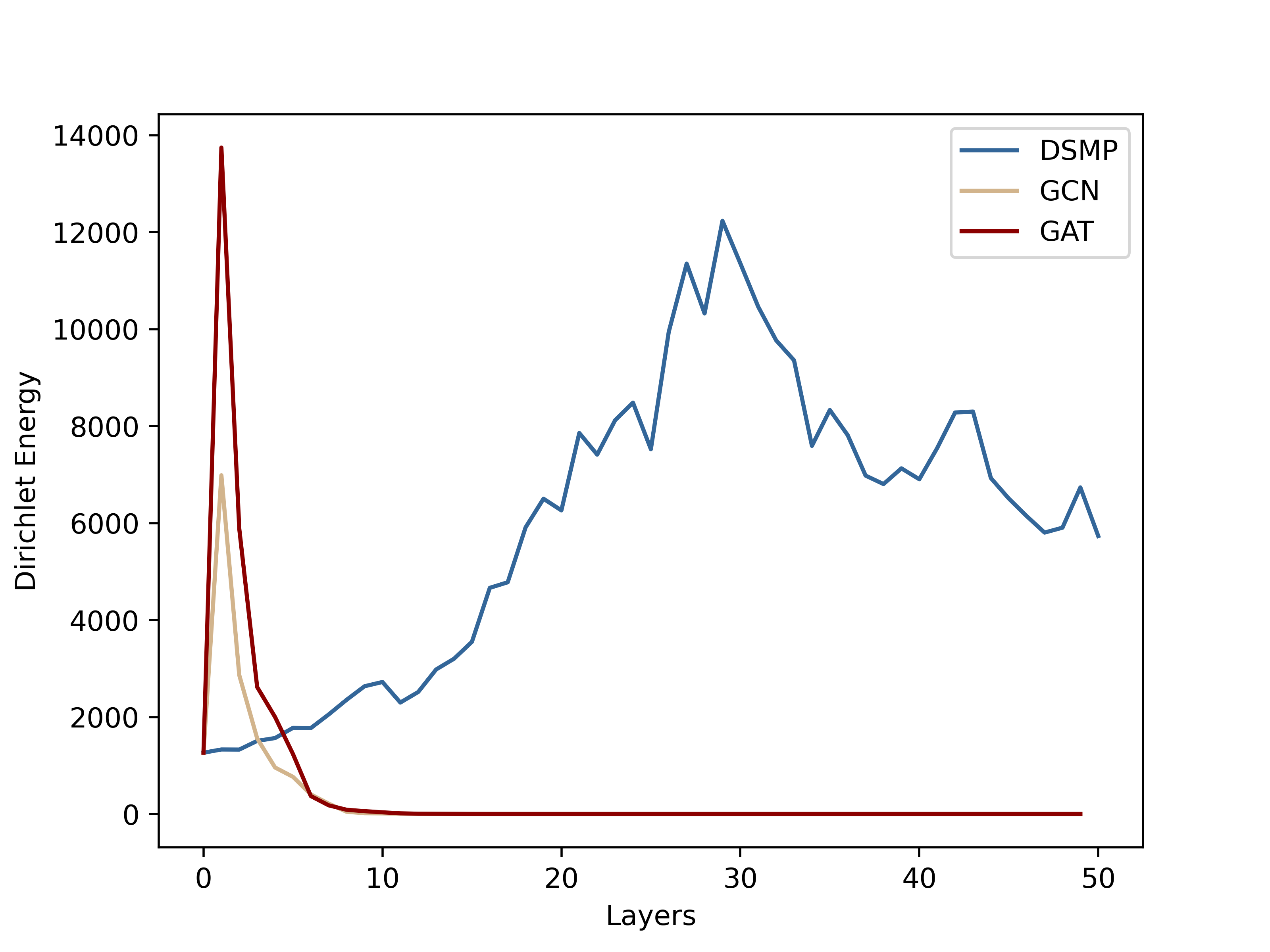}
\end{minipage}}
\centering
\subfigure[Accuracy vs Layers]{
\begin{minipage}[b]{0.45\linewidth}
\includegraphics[width=\textwidth]{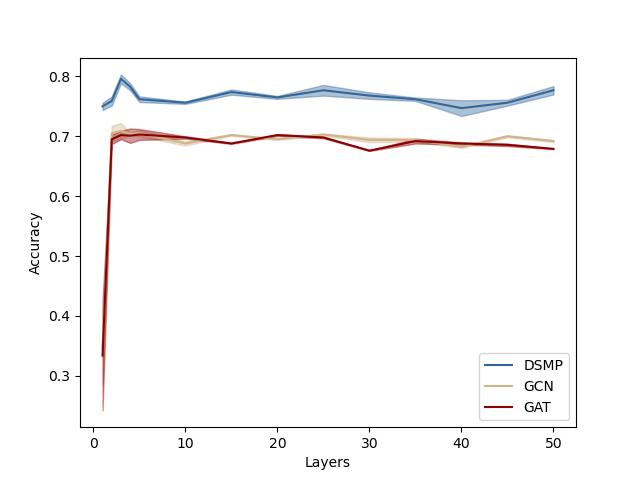}
\end{minipage}}
\caption{(a) We show the \textit{Dirichlet energy} of DSMP compared with several main-stream message passing models. (b) The model performance of graph classification on the PROTEINS dataset as the number of layers increases.}

\label{fig:dirichlet}
\end{figure*}

\subsection{Experiments and Analysis on Over-squashing Issue}
In recent research~\citep{karhadkar2022fosr}, the severity of the over-squashing issue within a graph structure is quantitatively assessed using the spectral gap, as outlined in section \ref{sec:gnn_with_os}. The spectral gap is intimately tied to the graph's connectedness, as the magnitude of its eigenvalues provides crucial insights. Specifically, the spectral gap can offer a proxy for structural bottlenecks through the Cheeger constant~\citep{chung1997spectral,cheeger2015lower,alon1984eigenvalues}. In essence, a larger spectral gap implies a graph with fewer structural bottlenecks, thereby indicating a reduced likelihood of encountering the over-squashing problem in MPNNs.

To empirically validate this, we computed the spectral gap for each graph instance in several publicly available datasets commonly used for graph classification tasks. Figure~\ref{fig:gap_range} presents a histogram of the spectral gaps, along with their mean values, providing a visual representation of the distribution and typical magnitudes observed across the datasets.

\begin{figure}[htb]
\centering
\subfigure[ENZYMES]{
\begin{minipage}[b]{0.45\linewidth}
\includegraphics[height=4.5cm]{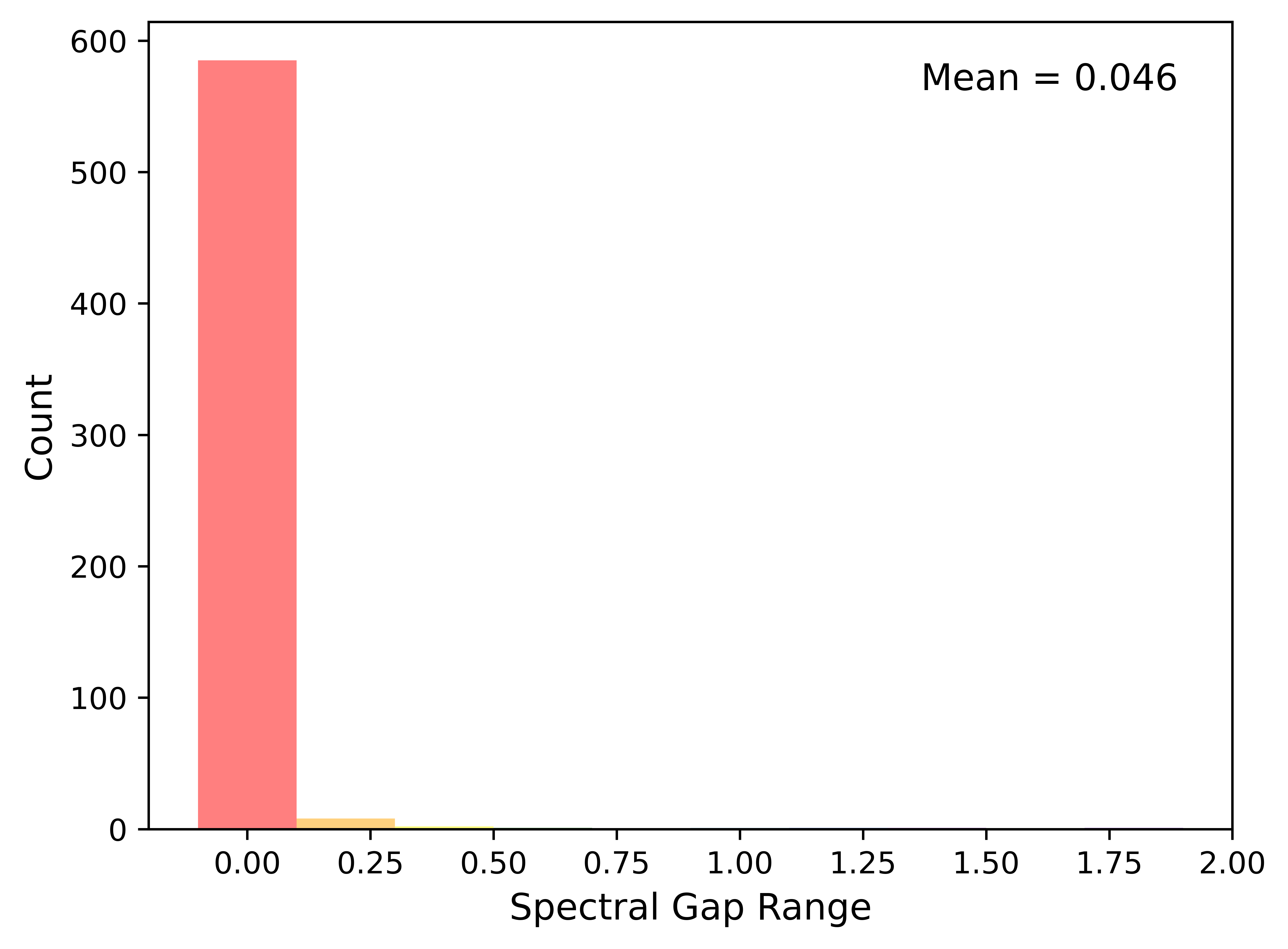}
\end{minipage}
}
\subfigure[IMDB]{
\begin{minipage}[b]{0.45\linewidth}
\includegraphics[height=4.5cm]{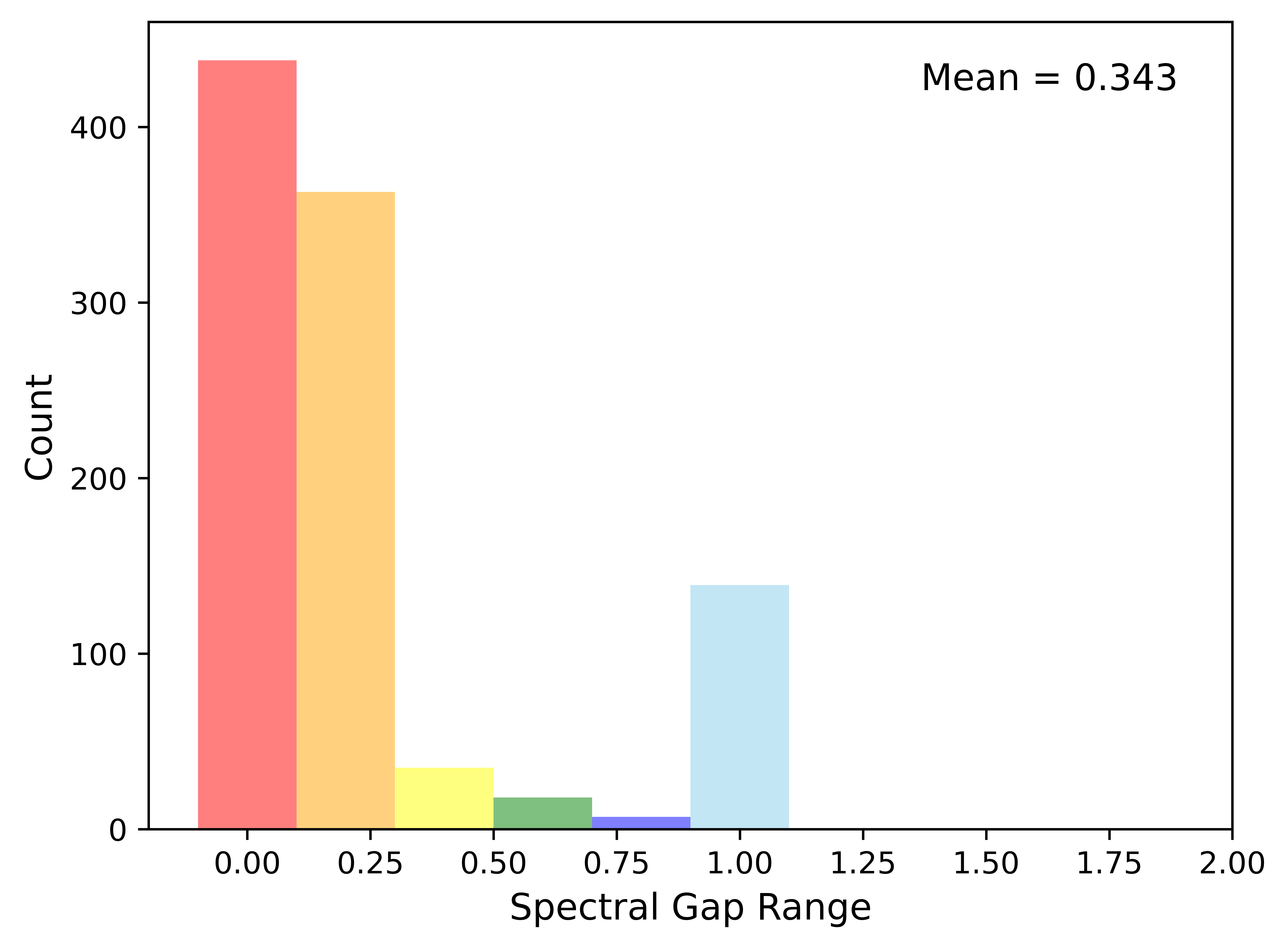}
\end{minipage}}
\subfigure[MUTAG]{
\begin{minipage}[b]{0.45\linewidth}
\includegraphics[height=4.5cm]{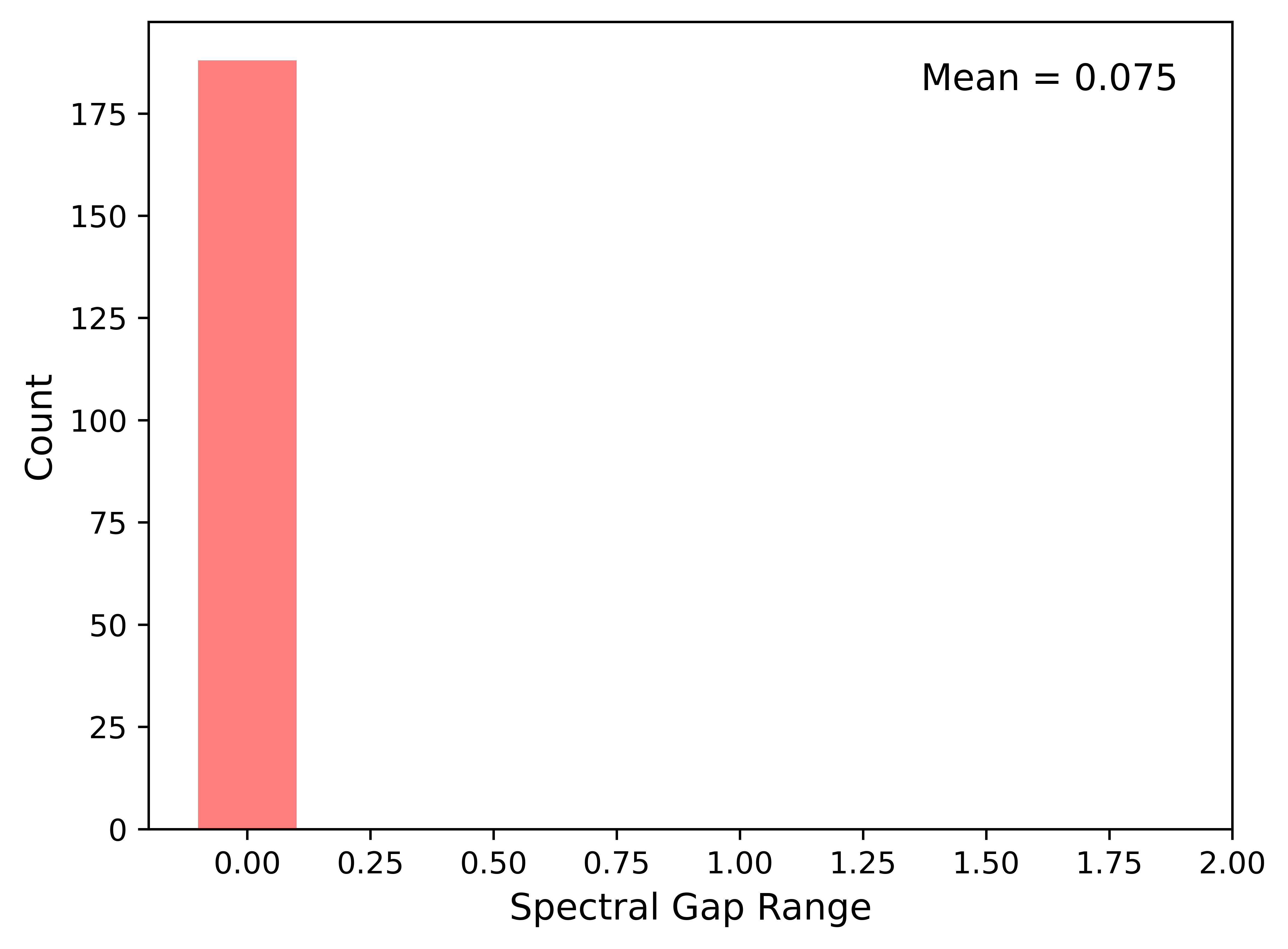}
\end{minipage}}
\subfigure[PROTEINS]{
\begin{minipage}[b]{0.45\linewidth}
\includegraphics[height=4.5cm]{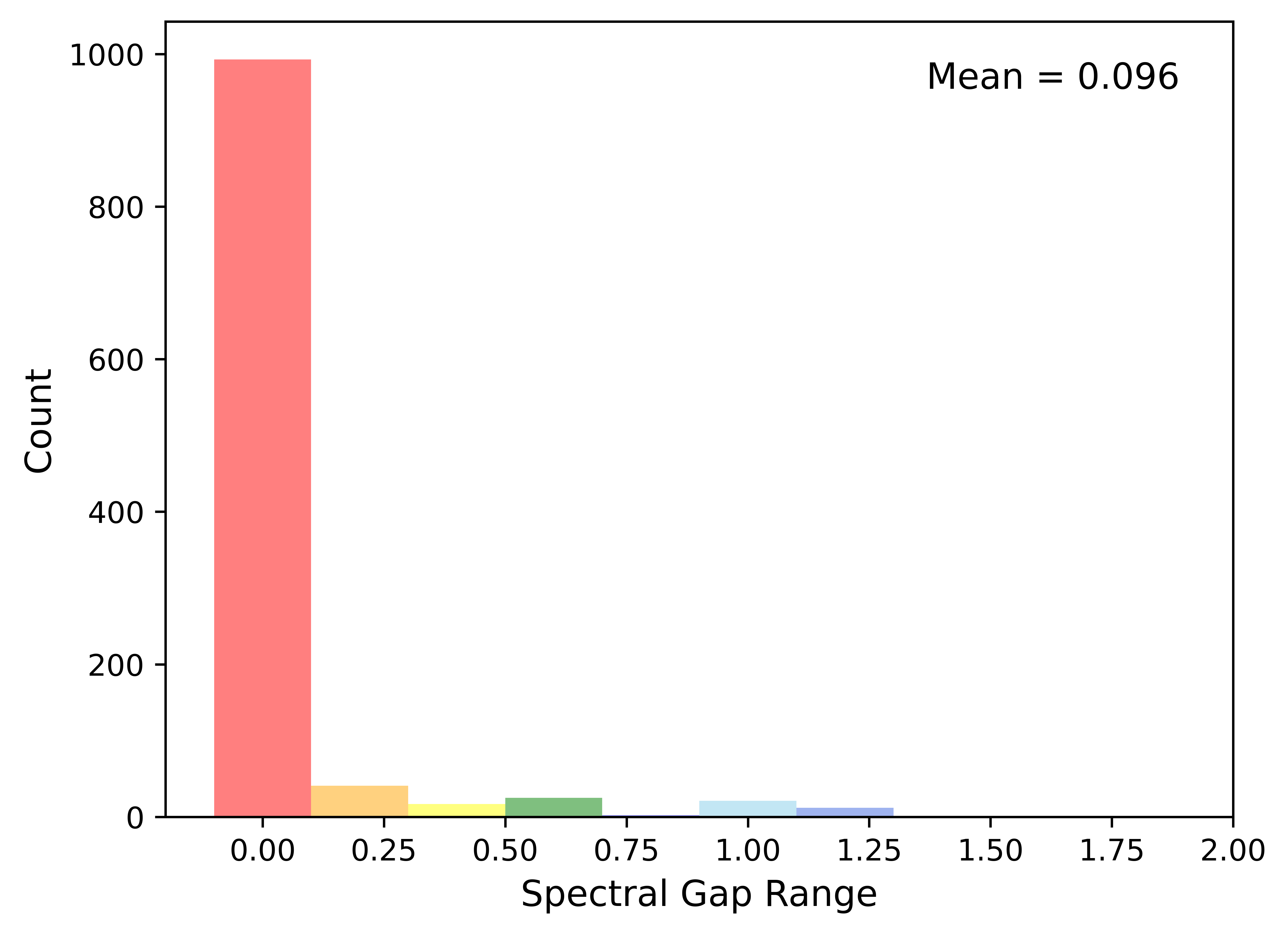}
\end{minipage}}
\caption{The measure of “bottleneckness” using spectral gap. We have quantified the number of graphs with spectral gaps occurring within each interval. The mean value of all spectral gaps is shown in the upper right corner.}
\label{fig:gap_range}
\end{figure}

Numerically, we compare the proposed DSMP model to no graph rewiring, as well as three SOTA baselines: SDRF~\citep{topping2022understanding}, FoSR~\citep{karhadkar2022fosr}, and BORF~\citep{nguyen2023revisiting}. We also evaluate the graph transformer model U2GNN~\citep{nguyen2022universal} to verify the impact of a fully connected graph structure.

We statistically measure the magnitude of the over-squashing issue on some public graph datasets and verify the effectiveness of DSMP, where the baseline rewiring methods have been previously tested~\citep{nguyen2022universal}. 
Compared to no rewiring GNN model, the baselines generally yield relatively significant gains except on IMDB as Table~\ref{tab:stats:os_classification} shows. By analyzing the “bottleneckness” from Figure~\ref{fig:gap_range} and Table~\ref{tab:stats:os_classification} results, we conclude these baseline rewiring methods are more impactful when inherent graph bottlenecks are severe. However, DSMP achieves the top performance across three graph classification tasks despite the over-squashing issue among these datasets. This indicates that DSMP more strongly addresses the over-squashing problem.

\begin{table*}[h]
\caption{The comparison with the graph rewiring methods on solving the over-squashing problem. The red color indicates that the rewiring model is superior to that of GCN addressing the ``bottleneckness'' issue among datasets, with darker shades representing more significant performance. The green color indicates the opposite case.}
\label{tab:stats:os_classification}
\begin{center}
 \begin{minipage}{0.95\textwidth}
\resizebox{\textwidth}{!}{
    \begin{tabular}{lccccccc}
    \toprule
    & \makebox[0.2\textwidth][c]{ \small \textbf{ENZYMES}} & \makebox[0.2\textwidth][c]{\small \textbf{IMDB}} &\makebox[0.2\textwidth][c]{\small \textbf{MUTAG}} & \makebox[0.2\textwidth][c]{\small \textbf{PROTEINS}}   \\
    \midrule
    \# graphs &$600$ & $1,000$ & $188$ & $1,113$ \\
    \# classes & $6$ & $2$ & $2$ & $2$ \\
    \# nodes & $2-126$ & $12-136$ & $10-28$ & $4-620$ \\
    Avg. \# nodes & $32.63$ & $19.77$ & $17.93$ & $39.06$ \\
    Avg. \# edges & $124.27$ & $193.062$ & $39.58$ & $145.63$ \\
    \midrule
    U2GNN & \cellcolor{white!50!red!100!}$37.5${\scriptsize$\pm1.8$} & \cellcolor{white!50!red!14!}$53.6${\scriptsize$\pm3.5$} & \cellcolor{white!50!red!33!}$\bf{88.5}${\scriptsize$\pm3.2$} &  \cellcolor{white!50!red!21!}$78.5${\scriptsize$\pm4.1$} \\ 
    \midrule
    GCN & $25.5${\scriptsize$\pm1.3$} & $49.3${\scriptsize$\pm1.0$} & $68.8${\scriptsize$\pm2.1$} & $70.6${\scriptsize$\pm1.0$} \\
    GCN+SDRF & \cellcolor{white!50!red!5!}$26.1${\scriptsize$\pm1.1$} & \cellcolor{white!50!green!1!}$49.1${\scriptsize$\pm0.9$} & \cellcolor{white!50!red!5!}$70.5${\scriptsize$\pm2.1$} &  \cellcolor{white!50!red!2!}$71.4${\scriptsize$\pm0.8$} \\ 
    GCN+FoSR & \cellcolor{white!50!red!10!}$27.4${\scriptsize$\pm1.1$} & \cellcolor{white!50!red!2!}$49.6${\scriptsize$\pm0.8$} & \cellcolor{white!50!red!14!}$75.6${\scriptsize$\pm1.7$} &  \cellcolor{white!50!red!3!}$72.3${\scriptsize$\pm0.9$} \\ 
    GCN+BORF & \cellcolor{white!50!green!20!}$24.7${\scriptsize$\pm1.0$} & \cellcolor{white!50!red!2!}$50.1${\scriptsize$\pm0.9$} & \cellcolor{white!50!red!1!}$75.8${\scriptsize$\pm1.9$} &  \cellcolor{white!50!green!4!}$71.0${\scriptsize$\pm0.8$} \\ 
    \midrule
    GIN & $31.3${\scriptsize$\pm1.2$} & $69.0${\scriptsize$\pm1.3$} & $75.5${\scriptsize$\pm2.9$} & $69.7${\scriptsize$\pm1.0$} \\
    GIN+SDRF & $33.5${\scriptsize$\pm1.3$} & $68.6${\scriptsize$\pm1.2$} & $77.3${\scriptsize$\pm2.3$} &  $72.2${\scriptsize$\pm0.9$} \\ 
    GIN+FoSR & $25.3${\scriptsize$\pm1.2$} & $69.5${\scriptsize$\pm1.1$} & $75.2${\scriptsize$\pm3.0$} &  $74.2${\scriptsize$\pm0.8$} \\ 
    GIN+BORF & $35.5${\scriptsize$\pm1.2$} & $71.3${\scriptsize$\pm1.5$} & $80.8${\scriptsize$\pm2.5$} &  $71.3${\scriptsize$\pm1.0$} \\ 
    
    \midrule
    DSMP & $\bf{40.6}${\scriptsize$\pm0.5$} & $\bf{77.3}${\scriptsize$\pm0.6$} & $85.0${\scriptsize$\pm0.25$} & $\bf{79.6}${\scriptsize$\pm0.7$} \\ 
    \bottomrule
    \end{tabular}
}
\end{minipage}
	\begin{minipage}{0.014\linewidth}
		\begin{annotate}{\includegraphics[width=0.8\linewidth]{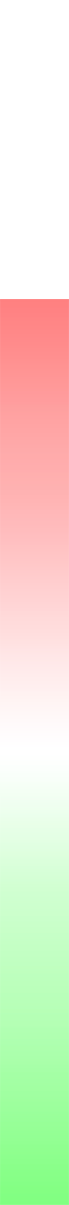}}{1}
    	\note{0.38,0.75}{- $100\%$}
    	\note{0.33,0.25}{- $50\%$}
    	\note{0.28,-0.35}{- $0\%$}
    	\note{0.4,-0.95}{- $-50\%$}
    	\note{0.45,-1.45}{- $-100\%$}
    	\end{annotate}
		\label{ }
	\end{minipage}
    \end{center}
\end{table*}

\subsection{Stability of DSMP}
To assess the stability of our proposed DSMP model in the semi-supervised node classification task, we conducted experiments on eight widely used benchmark datasets. The characteristics of these datasets are summarized in Table \ref{tab:stats:node_classification}. To introduce perturbations in the node attributes, we employed the \emph{attribute injection} method. This approach involves swapping attributes between distinct nodes in a random subgraph, similar to the technique described in \citep{ding2019deep}.

 A certain number of targeted nodes are randomly selected from the input graph. For each selected node $v_i$, we randomly pick $k$ nodes from the graph and select the node $v_j$ whose attributes deviate the most from node $v_i$ among the $k$ nodes by maximizing the Euclidean distance of node attributes, i.e.,  $\left\|x_{i}-x_{j}\right\|_{2}$. Then, we substitute the attributes $x_{i}$ of node $v_i$ with $x_{j}$. Specifically, we set the size of candidates $k=100$ for small datasets, i.e., Cora and Citeseer, and $k=500$ for relatively larger datasets, i.e., PubMed, Coauthor-CS, and Wiki-CS.

 By introducing these controlled perturbations, we aimed to evaluate the resilience of the DSMP model to attribute noise and assess its stability in node classification tasks. The results of these experiments will provide insights into the model's robustness and its ability to handle real-world scenarios with noisy or perturbed data.
\begin{table*}[!htbp]
\caption{Summary of the datasets for node classification tasks.}
\label{tab:stats:node_classification}
\begin{center}
\resizebox{\textwidth}{!}{
    \begin{tabular}{lcccccccc}
    \toprule
    & \textbf{Cora} & \textbf{Citeseer} & \textbf{PubMed} & \textbf{Wiki-CS} & \textbf{Coauthor-CS} & \textbf{Wisconsin} & \textbf{Texas} & \textbf{OGB-arxiv} \\
    \midrule
    \# Nodes & $2,708$ & $3,327$ & $19,717$ & $11,701$ & $18,333$ & $251$ & $183$ & $169,343$\\
    \# Edges & $5,429$ & $4,732$ & $44,338$ & $216,123$ & $100,227$ & $499$ & $309$ & $1,166,243$ \\
    \# Features & $1,433$ & $3,703$ & $500$ & $300$ & $6,805$ & $1,7033$ & $1,703$ & $128$ \\
    \# Classes & $7$ & $6$ & $3$ & $10$ & $15$ & $5$ & $5$ & $40$ \\
    \# Training Nodes & $140$ & $120$ & $60$ & $580$ & $300$  & $120$ & $87$ & $ 90,941$ \\
    \# Validation Nodes & $500$ & $500$ & $500$ & $1769$ & $200$ & $80$ & $59$ & $29,799$ \\
    \# Test Nodes & $1,000$ & $1,000$ & $1,000$ & $5847$ & $1000$ & $51$ & $37$ & $48,603$ \\
    Label Rate & $0.052$ & $0.036$ & $0.003$ & $0.050$ & $0.016$ & $0.478$ & $0.475$ & $0.537$ \\
    Feature Scale & $\{0,1\}$ & $\{0,1\}$ & $[0, 1.263]$ & $[-3,3]$ & $\{0,1\}$ & $\{0,1\}$ & $\{0,1\}$ & $[-1.389,1.639]$\\ 
    \bottomrule
    \end{tabular}
}
\end{center}
\end{table*}
To verify the stability of the proposed model, we compare the DSMP to three popular graph denoising models with different design philosophies: \textsc{APPNP} \citep{klicpera2018predict} avoids global smoothness with residual connections; \textsc{GNNGuard} \citep{zhang2020gnnguard} modifies neighbor relevance in message passing to mitigate local corruption; \textsc{ElasticGNN} \citep{liu2021elastic} and \textsc{AirGNN} \citep{liu2021graph} pursue local smoothness with sparseness regularizers. We also train a $2$ layer GCN \citep{kipf2017semi} as the baseline method, which smooths the graph signals by minimizing the Dirichlet energy in the spatial domain.

\begin{table*}[h]
    \caption{Average performance for node classification over $10$ repetitions. Due to a large number of nodes, the APPNP model encountered an out-of-memory issue on OGB-arxiv dataset.}
    \label{tab:node_classification}
    \resizebox{\textwidth}{!}{
    \begin{tabular}{lcccccccc}
    \toprule
    \textbf{Module} & \textbf{Cora} & \textbf{Citeseer} & \textbf{PubMed} & \textbf{Coauthor-CS} & \textbf{Wiki-CS} & \textbf{Wisconsin} & \textbf{Texas} & \textbf{OGB-arxiv}\\ 
    \midrule
    clean & $81.26${\scriptsize$\pm0.65$} &
    $71.77${\scriptsize$\pm0.29$} &
    $79.01${\scriptsize$\pm0.44$} &
    $90.19${\scriptsize$\pm0.48$} & 
    $77.62${\scriptsize$\pm0.26$} & 
    $56.47${\scriptsize$\pm5.26$} & $65.14${\scriptsize$\pm1.46$} &
    $71.10${\scriptsize$\pm0.21$} \\ 
    \midrule
    \textsc{GCN} & $69.06${\scriptsize$\pm0.74$} &
    $57.58${\scriptsize$\pm0.71$} &
    $67.69${\scriptsize$\pm0.40$} &
    $82.41${\scriptsize$\pm0.23$} & 
    $65.44${\scriptsize$\pm0.23$} & 
    $48.24${\scriptsize$\pm3.19$} &
    $58.92${\scriptsize$\pm2.02$} &
    \bf{$68.42${\scriptsize$\pm0.15$}}  \\     
    \textsc{APPNP} & $68.46${\scriptsize$\pm0.81$} &
    $60.04${\scriptsize$\pm0.59$} &
    $68.70${\scriptsize$\pm0.47$} &
    $71.14${\scriptsize$\pm0.54$} & 
    $56.53${\scriptsize$\pm0.72$} & 
    \bf{$61.76${\scriptsize$\pm5.21$}} &
    $59.46${\scriptsize$\pm0.43$} &
    OOM \\ 
    \textsc{GNNGuard} & $61.96${\scriptsize$\pm0.30$} &
    $54.94${\scriptsize$\pm1.00$} &
    $68.50${\scriptsize$\pm0.38$} &
    $80.67${\scriptsize$\pm0.88$} &
    $65.69${\scriptsize$\pm0.32$} & 
    $46.86${\scriptsize$\pm1.06$} &
    $59.19${\scriptsize$\pm0.81$} &
    $65.75${\scriptsize$\pm0.32$} \\ 
    \textsc{ElasticGNN} & \bf{$77.74${\scriptsize$\pm0.79$}} &
    \bf{$64.61${\scriptsize$\pm0.85$}} &
    $71.23${\scriptsize$\pm0.21$} &
    $79.91${\scriptsize$\pm1.39$} & 
    $64.18${\scriptsize$\pm0.53$} & 
    $53.33${\scriptsize$\pm2.45$} &
    $59.77${\scriptsize$\pm3.24$} &
    $41.34${\scriptsize$\pm0.38$}\\
    \textsc{AirGNN} & $76.22{\scriptsize\pm3.75}$ &
    $62.14{\scriptsize\pm0.82}$ &
    \bf{$74.73{\scriptsize\pm0.43}$} &
    $80.18{\scriptsize\pm0.31}$ & 
    $71.36{\scriptsize\pm0.20}$ & 
    $61.56{\scriptsize\pm0.72}$ &
    $59.46{\scriptsize\pm1.24}$ &
    $52.32{\scriptsize\pm0.58}$  \\
    \midrule
    \textsc{DSMP} & 
    $72.23{\scriptsize\pm0.34}$ & 
    $62.79{\scriptsize\pm0.55}$ &
    $72.53{\scriptsize\pm0.82}$ & 
    $\bm{82.70}{\scriptsize\pm0.37}$ & 
    $\bm{75.54{\scriptsize\pm0.27}}$ & 
    $52.94{\scriptsize\pm1.9}$ & 
    $\bm{60.36{\scriptsize\pm2.67}}$ & 
    $66.73{\scriptsize\pm0.56}$ \\
    \bottomrule\\[-2.5mm]
    \end{tabular}
 }
\end{table*}

\subsection{Graph Classification Experiments}
In this section, we compare the DSMP model with baseline models of graph representation learning. Following the convention, we report the percentage value of mean test accuracy for the classification tasks with TUDatasets and ROC-AUC score for the ogbg-molhiv dataset. The mean performance scores are averaged over 10 repetitions, with their standard deviations provided after the $\pm$ signs. As indicated in Table~\ref{tab:stats:graph_classification}, the DSMP model outperforms the other baseline models in general.

Compared with GCN, the DSMP model outperforms by 12.8\% at most on the PROTEINS dataset, and by 0.9\% at least on the QM7 dataset. When compared with GIN, our model outperforms by 14.2\% at most on the PROTEINS dataset, and by 0.24\% at least on the COLLAB dataset. The DSMP model consistently demonstrates superior performance than the HDS-GAT model, further validating its effectiveness and reliability. Our method shows a noticeable improvement compared to GCN on most datasets, except for QM7 and COLLAB. As illustrated in Figure~\ref{fig:freq_range}, the observed performance difference can be attributed to the fact that these two datasets contain limited high-frequency information. Consequently, methods designed to extract high-frequency features may not be as effective in this context.

\begin{table*}[h]
\caption{Summary of the datasets for graph-level classification tasks. QM7 dataset is not a typical dataset for graph-level classification and we use $R$ to denote it for a regression task. }
\label{tab:stats:graph_classification}
\begin{center}
\resizebox{\textwidth}{!}{
    \begin{tabular}{lccccccc}
    \toprule
    & \textbf{QM7} & \textbf{PROTEINS} & \textbf{D\&D} & \textbf{Mutagenicity} & \textbf{COLLAB} & \textbf{OGBG-MOLHIV} & \textbf{NCI1}  \\
    \midrule
    \# graphs &$7,165$ & $1,113$ & $1,178$ & $4,337$ & $5,000$ & $41,127$  & $4,110$ \\
    \# classes & 1 ($R$) & $2$ & $2$ & $2$ & $3$ & $2$ & $2$ \\
    \# nodes & $4-23$ & $4-620$ & $30-5,748$ & $30-417$ & $32-492$ & $2-222$ & $3-111$ \\
    Avg. \# nodes & $15$ & $39$ & $284$ & $30$ & $74$ & $26$ & $30$ \\
    Avg. \# edges & $123$ & $73$ & $716$ & $31$ & $2,458$ & $28$ & $32$ \\
    \# Features  & $0$ & $3$ & $89$ & $14$ & $0$  & $9$ & $37$ \\
    \midrule
    GCN & $41.9${\scriptsize$\pm1.1$} & $70.6${\scriptsize$\pm1.0$} & $78.9${\scriptsize$\pm3.3$} & $80.7${\scriptsize$\pm3.3$} & $81.8${\scriptsize$\pm1.3$} & $74.9${\scriptsize$\pm2.6$} & $76.9${\scriptsize$\pm1.7$} \\
    GIN & $43.0${\scriptsize$\pm1.5$} & $69.7${\scriptsize$\pm1.0$} & $78.6${\scriptsize$\pm1.3$} & $81.3${\scriptsize$\pm1.1$} & $82.3${\scriptsize$\pm0.9$} & $77.4${\scriptsize$\pm1.2$} & $78.8${\scriptsize$\pm1.5$} \\
    UFG & $41.7${\scriptsize$\pm0.8$} & $77.8${\scriptsize$\pm2.6$} & $80.9${\scriptsize$\pm1.7$} &  $81.6${\scriptsize$\pm1.4$} &
    $\bf{82.8}${\scriptsize$\pm0.3$} & $78.8${\scriptsize$\pm0.5$} & $77.9${\scriptsize$\pm1.2$}  \\
    HDS-GAT & $42.2${\scriptsize$\pm0.5$} & $78.5${\scriptsize$\pm1.5$} & $79.8${\scriptsize$\pm1.1$} &  $81.2${\scriptsize$\pm0.5$} &
    $82.2${\scriptsize$\pm0.6$} & $78.3${\scriptsize$\pm0.8$} & $78.2${\scriptsize$\pm1.5$}  \\
    DSMP & $\bf{41.5}${\scriptsize$\pm1.0$} & $\bf{79.6}${\scriptsize$\pm0.7$} & $\bf{82.5}${\scriptsize$\pm1.1$} & $\bf{82.0}${\scriptsize$\pm0.9$} & $82.5${\scriptsize$\pm0.3$} & $\bf{79.2}${\scriptsize$\pm1.7$} & $\bf{79.8}${\scriptsize$\pm1.0$} \\ 
    \bottomrule
    \end{tabular}
}
\end{center}
\end{table*}

\begin{figure}[h]
\centering
\subfigure[QM7]{
\begin{minipage}[b]{0.45\linewidth}
\includegraphics[height=4.2cm]{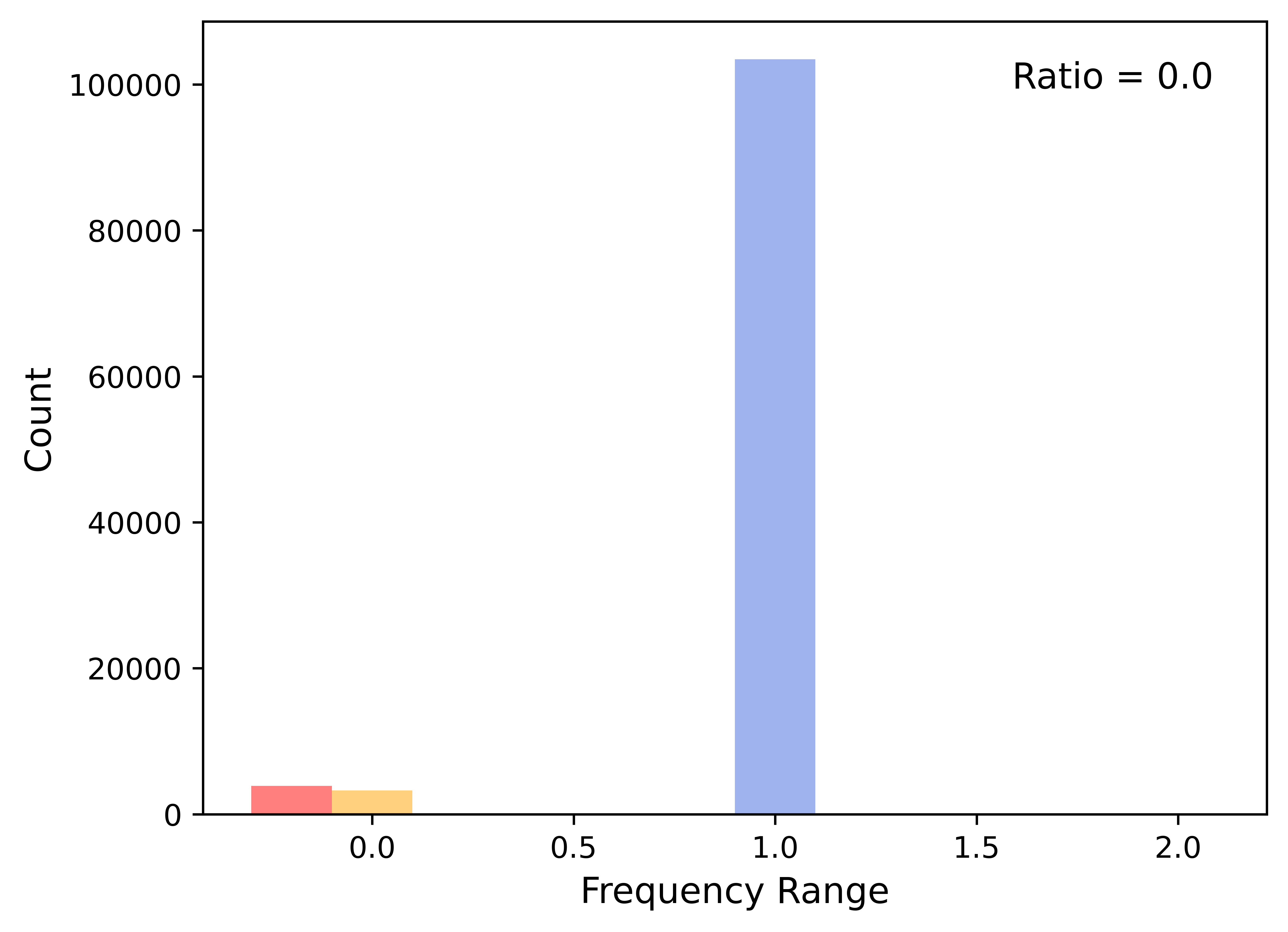}
\end{minipage}
}
\subfigure[COLLAB]{
\begin{minipage}[b]{0.45\linewidth}
\includegraphics[height=4.2cm]{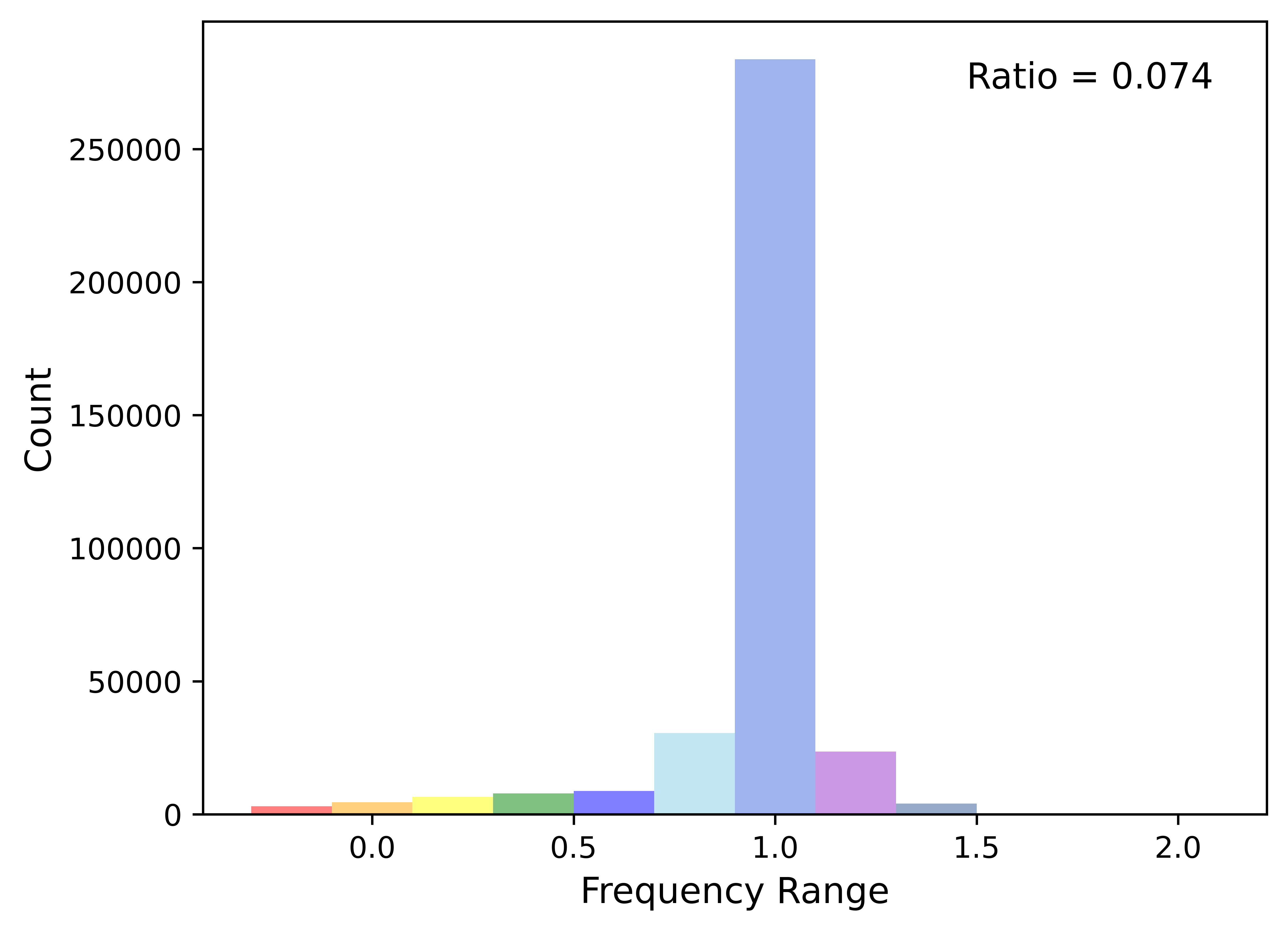}
\end{minipage}}
\subfigure[PROTEINS]{
\begin{minipage}[b]{0.45\linewidth}
\includegraphics[height=4.2cm]{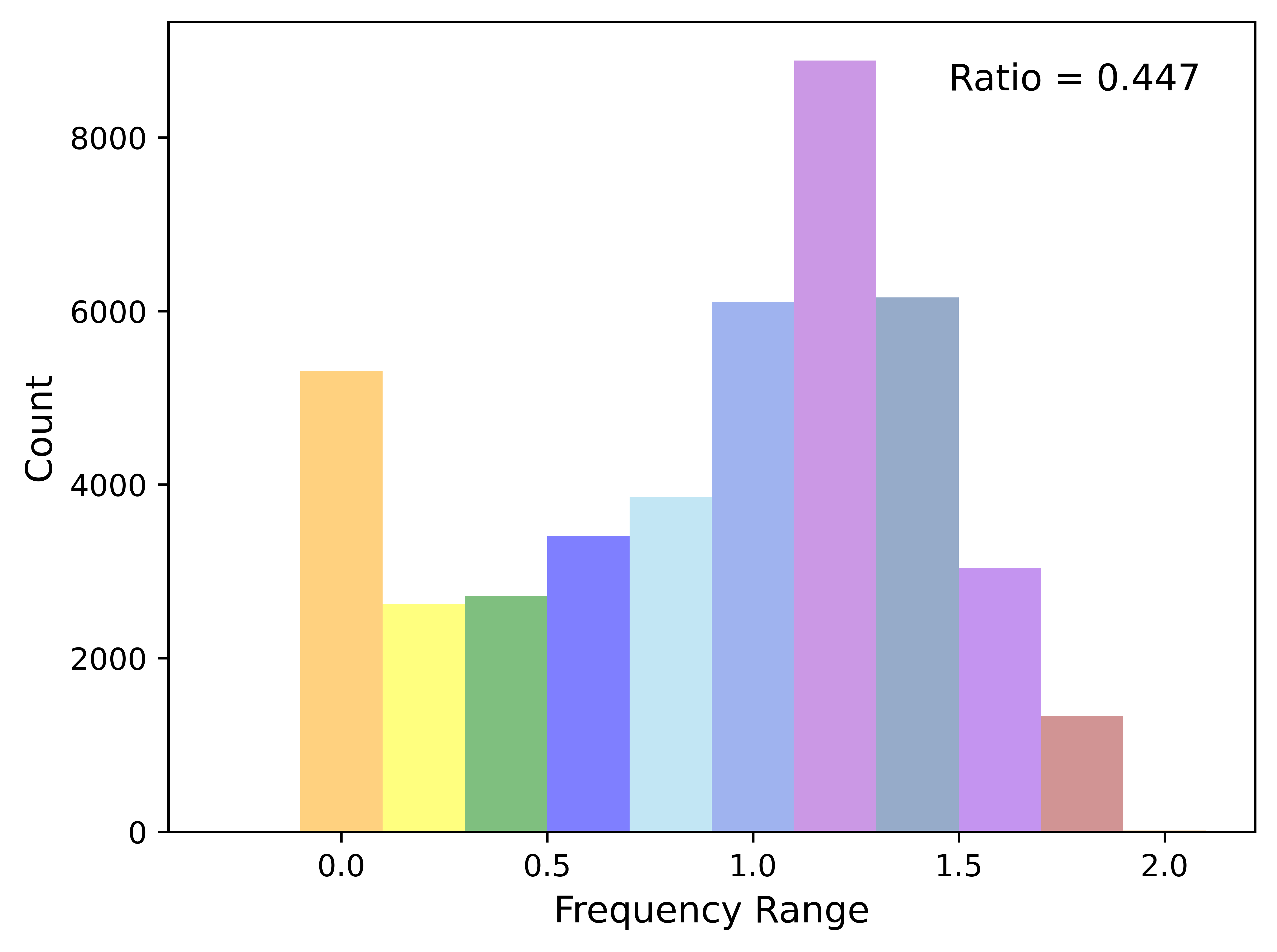}
\end{minipage}}
\subfigure[D\&D]{
\begin{minipage}[b]{0.45\linewidth}
\includegraphics[height=4.2cm]{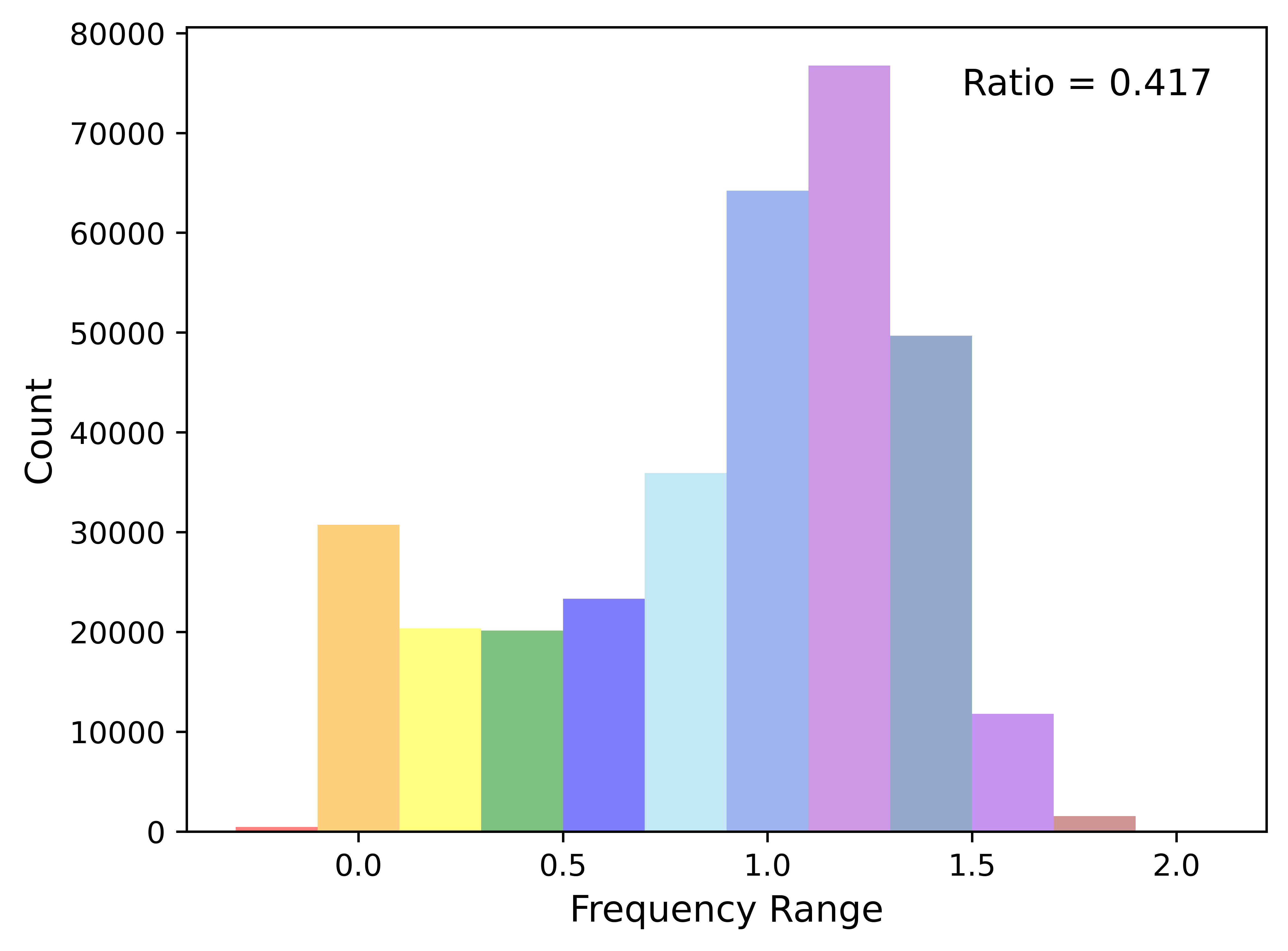}
\end{minipage}}
\subfigure[Mutagenicity]{
\begin{minipage}[b]{0.45\linewidth}
\includegraphics[height=4.2cm]{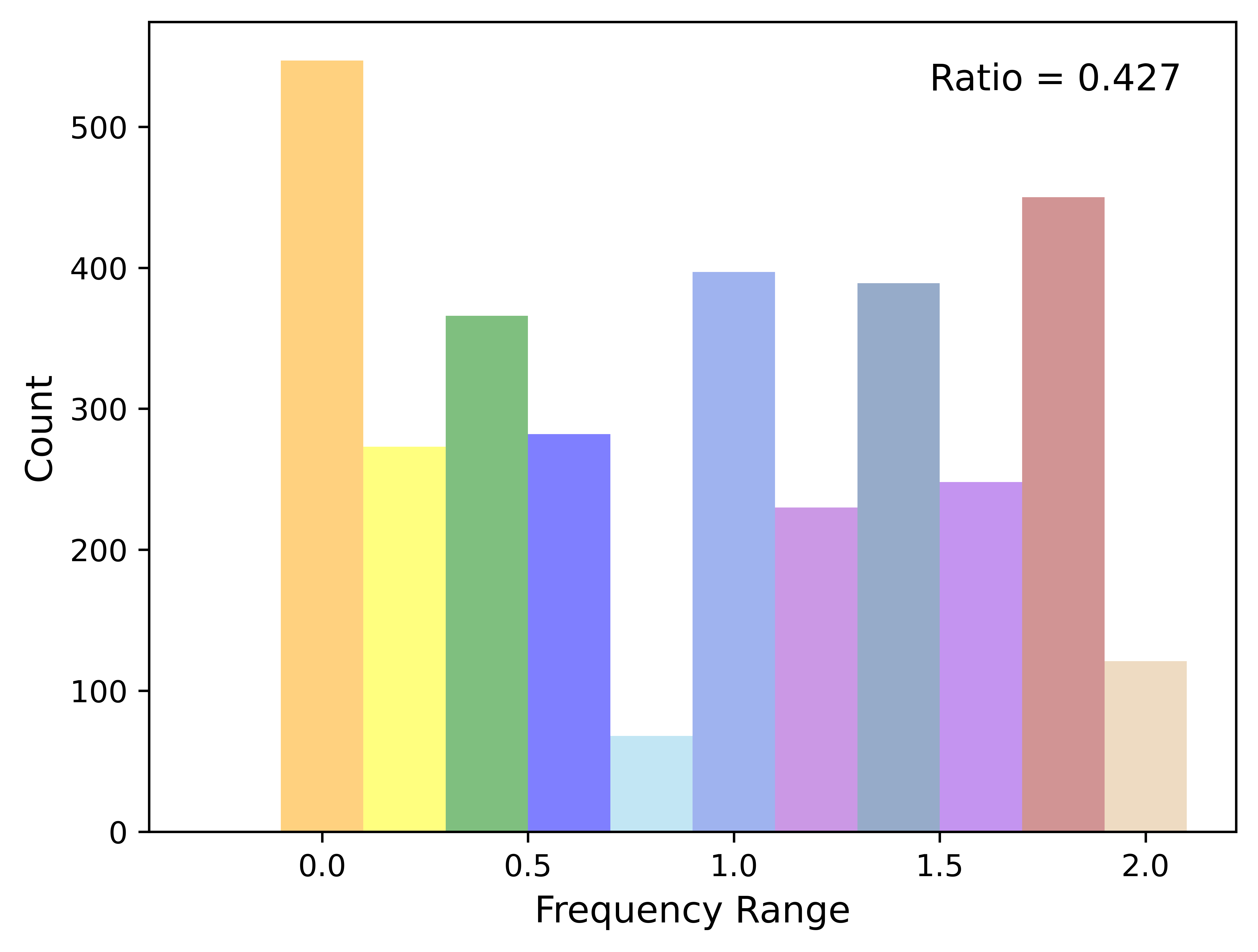}
\end{minipage}}
\subfigure[OGBG-MOLHIV]{
\begin{minipage}[b]{0.45\linewidth}
\includegraphics[height=4.2cm]{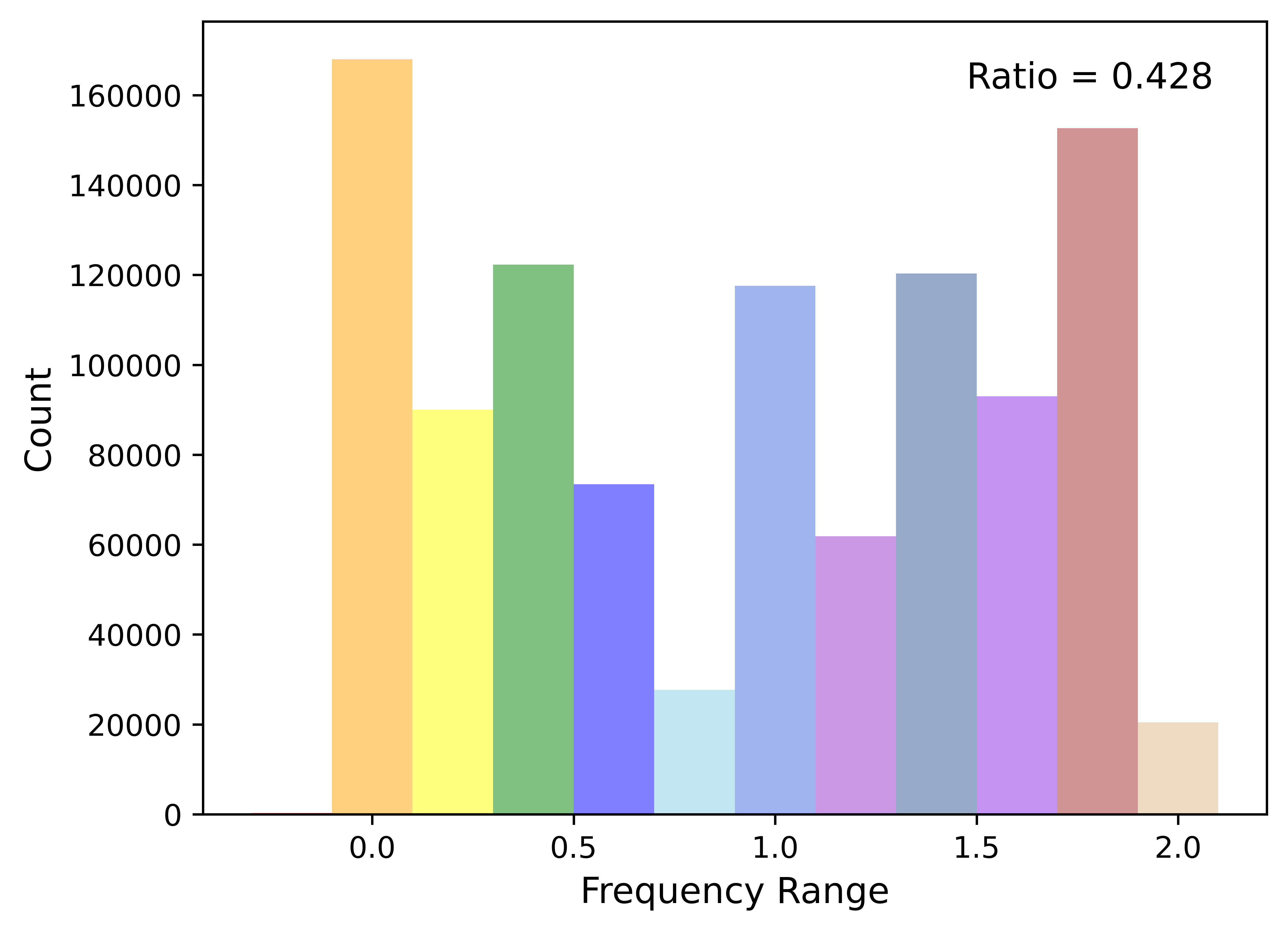}
\end{minipage}}
\subfigure[NCI1]{
\begin{minipage}[b]{0.45\linewidth}
\includegraphics[height=4.2cm]{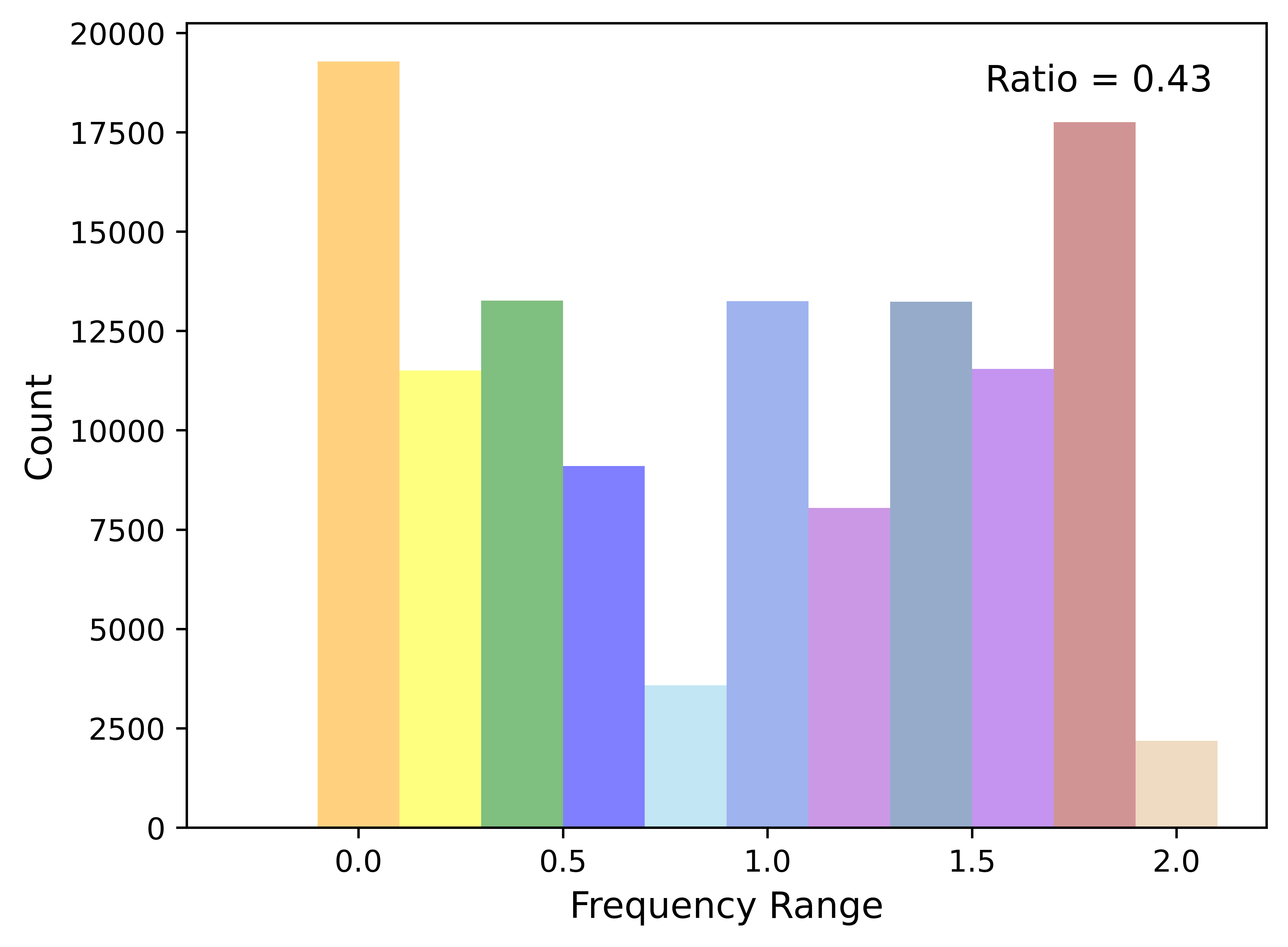}
\end{minipage}}
\subfigure[ENZYMES]{
\begin{minipage}[b]{0.45\linewidth}
\includegraphics[height=4.2cm]{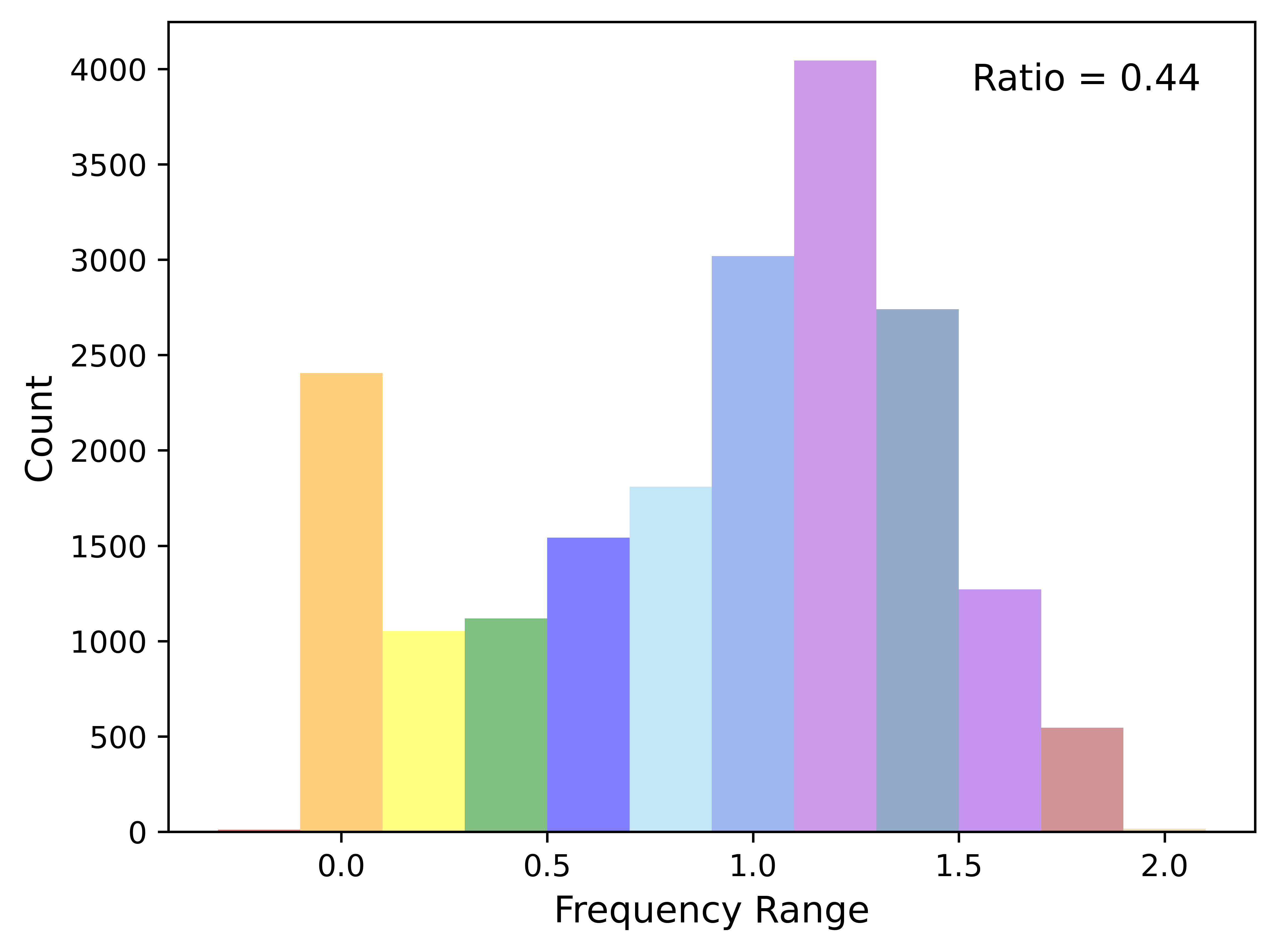}
\end{minipage}}
\caption{The frequency analysis of different graph-level classification datasets. We count the number of frequencies in each interval and display the ratio of frequency greater than 1.2 in the upper right corner of the each figure.}
\label{fig:freq_range}
\end{figure}

\subsection{Computational Cost}
\paragraph{Time complexity}
Let $E=|\mathcal{E}|$ be the number of edges in the graph, $N$ be the number of nodes and $d$ be the feature dimension of nodes. According to \eqref{eq:mpnn}, the computational cost for one propagation is $O\left(ENd\right)$. For the propagation of the SMP module, as \eqref{eq:smpnn} shows, the computational cost is also $O\left(ENd\right)$.

We also compare the computational cost of the DSMP model with some popular MPNNs, which are considered as baseline models, in terms of testing time complexity. In Table \ref{tab:computationalTime}, we present the time cost for each model. We randomly select a batch of data with $8$ graph samples from the PROTEINS dataset and increase the number of layers in the message passing frameworks.
\begin{table}[h]
    \caption{Comparison of the time cost on PROTEINS dataset.}
    \centering
    \resizebox{0.6\linewidth}{!}{
    \begin{tabular}{lcccc}
        \toprule
        \makebox[0.2\textwidth][l]{\textbf{Layer Num}} &  \makebox[0.2\textwidth][c]{$1$ } & \makebox[0.2\textwidth][c]{$3$} & \makebox[0.2\textwidth][c]{$50$} \\
        \midrule
        \textsc{GCN} & $ 0.0847$ &$0.0868$ & $0.1151$ \\
        \textsc{GIN} & $ 0.0857$ &$0.088$ & $0.1444$  \\
        \textsc{DSMP} & $ 0.0992$ &$0.1226$ & $0.3048$  \\
        \bottomrule
    \end{tabular}
    }
    \label{tab:computationalTime}
\end{table}

\paragraph{Space complexity}
In conventional message passing frameworks, the sparse edge connection matrix and dense node feature matrix are stored in memory, resulting in a memory cost of $2E + O(Nd)$. According to \eqref{eq:matrix_form}, the propagation of each SMP module requires one low-pass framelet matrix and $K \times J$ high-pass framelet matrices. In our numerical experiments, we set $K=1$ and $J=2$ to utilize scaling information up to $2$, which connects only 2-hop neighbors as shown in Figure~\ref{fig:main_idea}. Let $R$ be the number of unconnected 2-hop neighbors, then the memory cost for SMP is $6E + 2R + O(Nd)$.

\section{Conclusion and Future Work}
In conclusion, this work introduces the DSMP neural network as a solution to address the stability, over-smoothing, and over-squashing issues commonly observed in current GNN models. By integrating the multiscale and multilevel scattering transform with a learnable message passing scheme, the DSMP model offers a novel approach to processing graph-structured data.
The proposed message passing framework, viewed from the perspective of spectral methods, has been theoretically proven to effectively mitigate the aforementioned problems. Through extensive experiments, we have demonstrated that the DSMP model outperforms existing GNNs in various graph reasoning tasks, highlighting its superior performance and effectiveness.

In terms of future work, there are several potential directions to explore. Firstly, further investigation can be conducted to enhance the scalability and efficiency of the DSMP model, particularly in large-scale graph datasets. Additionally, exploring the applicability of the DSMP model in other domains, such as social network analysis or recommendation systems, could provide valuable insights into its versatility and generalizability.
Furthermore, the interpretability of the DSMP model can be explored, enabling a better understanding of its decision-making process. This could involve analyzing the importance of different features or nodes in the graph and providing explanations for the model's predictions.

Overall, the proposed DSMP model presents a promising direction for improving the performance and addressing the limitations of GNNs. The future work mentioned above will contribute to further advancing the field of GNNs and their applications.

\vskip 0.2in
\newpage
\bibliography{reference.bib}

\appendix

\section{Table of notations}
\label{app:notations}

\scalebox{.9}{
\begin{tabular}{l|p{0.81\textwidth}}
\toprule
\textbf{Symbol} & \textbf{Meaning}\\
\toprule
$\mathcal{G}$ & graph\\
$\mathcal{V}$ & set of nodes (vertices) \\
$\mathcal{E}$ & set of edges \\
$\mathcal{A}$ & adjacency matrix \\
$\mathcal{N}^{m}(i)$ & $m$-hop neighbors of $i$-th node \\
$\mathcal{N}(i)$ &  $1$-hop neighbors of $i$-th node \\
$|\mathcal{V}|$ & cardinality of set $\mathcal{V}$\\
$\operatorname{vol}(\mathcal{G})$ & volume of a graph\\
$\boldsymbol{d}(v)$ & degree of vertex $v$\\
$ \mathbb{R}$ & real vector space \\
$l_2(\mathcal{G})$  & Hilbert space on graph $\mathcal{G}$\\
$L$ & graph Laplacian matrix \\
$\mathcal{L}$ & normalized graph Laplacian matrix \\
$\mathcal{D}$ &  degree matrix \\
$\delta_{\ell, \ell^{\prime}}$ & Kronecker delta function \\
$a$ & low-pass filter\\
$b^{(r)}$ & $r$-th high-pass filter\\
$\widehat{\alpha}, \widehat{\beta}$ & Fourier transformation of $a, b$\\
$\boldsymbol{\eta}$ & filter bank ( set of filtering functions ) \\
$\Psi$ & set of scaling functions \\
$R$    & dilation scale \\
$\{(\lambda_\ell,\vu_\ell)\}_{\ell=1}^{n}$ & eigenpairs of the graph Laplacian $\mathcal{L}$\\
$\boldsymbol{\varphi}_{l,p}(v)$, $\boldsymbol{\psi}_{l,p}^{(r)}(v)$ &  low-pass and the $r$-th high-pass framelet basis of node $v$ on level $l$ and node $p$\\
$\boldsymbol{\gW}_{k,l}$ & framelet decomposition operator of scale $k$ and level $l$\\
$\boldsymbol{\gW}_{k,l}^{\natural}$ & approximated framelet decomposition operator\\
$\boldsymbol{X}^{\text {in }}$ & input graph feature matrix\\
$\boldsymbol{X}^{\text {out }}$ & embedding of graph feature matrix\\
$\hat{\boldsymbol{X}}$ & framelet coefficients \\
$\Phi_{\mathcal{G}}(\mX)$ & graph scattering transformation of $\mathcal{G}$\\
$\phi_m(\mathcal{G}, \mX)$ & $m$-th layer of scattering transformation \\
$\varphi$ & average operator \\
$\hat{\sigma}$ & pointwise absolute value function\\
$\sigma$ & pointwise ReLU function \\
$\square(\cdot)$ & aggregation function\\
$\Gamma(\cdot)$, $\phi(\cdot)$ & linear functions\\
$E_{\mathcal{G}}(\mX)$ & \textit{Dirichlet energy} of graph $\mathcal{G}$ with feature matrix $\mX$ \\
$h(\mathcal{G})$ & Cheeger constant\\
$\partial \mathcal{S}$ & edge boundary of subset $\mathcal{S}$\\
$\boldsymbol{Z}^{(t)}$ & scattering coefficients at layer $t$\\
$\boldsymbol{\theta}_{0}, \boldsymbol{\theta}_{r,l}$ & learnable parameter matrices\\
$\tr(\mX)$ & trace of matrix $\mX$\\
$\mu_i(\mathcal{A})$ & $i$-th largest eigenvalue of matrix $\mathcal{A}$\\
$\mathcal{R}(\mathcal{G})$ & rewired graph of $\mathcal{G}$\\

\bottomrule
\end{tabular}
}

\end{document}

%% file: math_commands.tex
\usepackage{amsmath,amsfonts,bm}

\newcommand{\tr}{\mathrm{tr}}

\newcommand{\uG}{p}

\newcommand{\cfra}[1][l,\uG]{ 
    \boldsymbol{\varphi}_{#1}
    }
\newcommand{\cfrb}[2][l,\uG]{ 
    \boldsymbol{\psi}_{#1}^{#2}
}
\newcommand{\scala}{ 
    \alpha
}
\newcommand{\scalb}{ 
    \beta
}

\newcommand{\ipG}[1]{\left\langle #1 \right\rangle
}

\newcommand{\FT}[1]{ 
    \widehat{#1}
}

\newcommand{\conj}[1]{ 
 \overline{#1}
}

\def\1{\bm{1}}

\def\vu{{\bm{u}}}

\def\mI{{\bm{I}}}

\def\mU{{\bm{U}}}

\def\mX{{\bm{X}}}
\def\mY{{\bm{Y}}}
\def\mZ{{\bm{Z}}}

\DeclareMathAlphabet{\mathsfit}{\encodingdefault}{\sfdefault}{m}{sl}
\SetMathAlphabet{\mathsfit}{bold}{\encodingdefault}{\sfdefault}{bx}{n}

\def\gG{{\mathcal{G}}}

\def\gT{{\mathcal{T}}}

\def\gV{{\mathcal{V}}}
\def\gW{{\mathcal{W}}}

\newcommand{\R}{\mathbb{R}}

